\documentclass{amsart}
\usepackage{amssymb}
\usepackage{graphicx}

\newtheorem{theorem}{Theorem}
\newtheorem{definition}{Definition}
\newtheorem{proposition}{Proposition}
\newtheorem{lemma}{Lemma}

\newtheorem{exam}{Example}

\newtheorem{exams}{Examples}

\newtheorem{rmk}{Remark}
\newenvironment{remark}{\begin{rmk}\rm}{\end{rmk}}
\newtheorem{notat}{Notation}
\newenvironment{notation}{\begin{notat}\rm}{\end{notat}}

\renewcommand{\ne}{\not =}

\title{Local Brunella's Alternative I.  RICH Foliations}
\author{F. Cano, M. Ravara-Vago \& M. Soares}

\address{Felipe Cano,
Dep. Algebra, An\'alisis Matem\'atico y Geometr\'ia y Topolog\'ia - UVa -
Valladolid, Espa\~na.}
\email{fcano@agt.uva.es}

\address{Marianna Ravara-Vago,
Dep. Matem\'{a}tica
ICEx - UFMG, Campus Pampulha
31270-901 Belo Horizonte - Brasil.
}\email{ravaravago@gmail.com}

\address{M\'arcio Soares,
Dep. Matem\'{a}tica
ICEx - UFMG, Campus Pampulha
31270-901 Belo Horizonte - Brasil.
}\email{msoares@mat.ufmg.br}

\subjclass{32S45, 32S65. Partially supported by MIC (Spain), CNPq, CAPES and FAPEMIG, (Brasil)}

\date{September 2013}

\begin{document}

\maketitle

\begin{flushright}
\em To Marco Brunella, in memoriam.
\end{flushright}

\tableofcontents

\begin{abstract}
This paper is devoted to studying the structure of codimension one singular holomorphic foliations on $({\mathbb C}^3,0)$ without invariant germs of analytic surface. We focus on the so-called CH-foliations, that is, foliations without saddle nodes in two dimensional sections. Considering a reduction of singularities, we detect the possible existence of ``nodal components'', which are a higher dimensional version of the nodal separators in dimension two. If the foliation is without nodal components, we prove that all the leaves in a neighborhood of the origin contain at least one germ of analytic curve at the origin. We also study the structure of nodal components for the case of ``Relatively Isolated CH-foliations'' and we show that they cut the dicritical components or they exit the origin through a non compact invariant curve. This allows us to give a precise statement of a local version of Brunella's alternative: if we do not have an invariant surface, all the leaves contain a germ of analytic curve or it is possible to detect the nodal components in the generic points of the singular curves before doing the reduction of singularities.
\end{abstract}

\section{Introduction}

It is a question of M. Brunella to decide if the following alternative is true:
\begin{quote}\em
Let $\mathcal F$ be a singular holomorphic foliation of codimension one in the projective space ${\mathbb P}^3_{\mathbb C}$. If there is no projective algebraic surface invariant by $\mathcal F$, then each leaf of $\mathcal{F}$ is a union of algebraic curves.
\end{quote}

The answer to this question is known \cite{Cer} to be positive in the case of generic foliations in a pencil of foliations.  For degree $\mathsf{d} =0,1$ and $2$ all the irreducible components of the space of foliations $\mathcal{F}(3, \mathsf{d})$ are known but, for $\mathsf{d} \geq 3$, although several irreducible components have been recognized, it is not known if this list is exhaustive. What is known is that some irreducible components of  $\mathcal{F}(3, \mathsf{d})$ admit such pencils, hence the positive answer to the alternative in these cases.

This paper is the first one concerning a local version of the above alternative for complex hyperbolic foliations on $({\mathbb C}^3,0)$. As we state in Definition \ref{def:CHfoliaiton}, a  germ $\mathcal F$ of singular holomorphic foliation of codimension one in $({\mathbb C}^n,0)$ is a {\em complex hyperbolic foliation} (for short,  a {\em CH-Foliation}) if for every holomorphic map germ
$
\phi:({\mathbb C}^2,0)\rightarrow ({\mathbb C}^n,0)
$
generically transversal to $\mathcal F$, the transformed foliation $\phi^*{\mathcal F}$ is a generalized curve on $({\mathbb C}^2,0)$ in the sense of \cite{Cam-N-S}; that is, there are no saddle nodes in its reduction of singularities. As there are dicritical CH-foliations without invariant surfaces, this phenomenon warns against the use of the terminology ``generalized surface'' for  dicritical situations. In contrast with this, the authors in \cite{Moz-F} use the term ``generalized surface'' in the non-dicritical case, since the reduction of singularities of the set of invariant surfaces provides a reduction of singularities for the foliation.

The first result we prove in this paper is the following one
\begin{theorem}
\label{teo:mainI}
 Let $\mathcal F$ be a CH-foliation on $({\mathbb C}^3,0)$ without germ of invariant surface. Assume that there is a reduction of singularities of $\mathcal F$ without nodal components. There is a neighborhood $U$ of the origin $0\in {\mathbb C}^3$ such that, for each leaf $L\subset U$ of $\mathcal F$ in $U$ there is a germ of analytic curve $\gamma$ at the origin such that $\gamma\subset L\cup\{0\}$.
\end{theorem}

We know by  \cite{Can} that there is a reduction of singularities for any codimension one foliation $\mathcal F$ on $({\mathbb C}^3,0)$. That is, there is a morphism
$$
\pi:(M,\pi^{-1}(0))\rightarrow ({\mathbb C}^3,0)
$$
which is a composition of blow-ups with invariant centers that produces a normal crossings exceptional divisor $E\subset M$, in such a way that all the points $p\in \pi^{-1}(0)$ are simple points for the pair $\pi^*{\mathcal F}, E$. The simple points for CH-foliations are of a special type that we call {\em simple CH-points} (in dimension two this corresponds exactly to avoiding saddle nodes in the reduction of singularities as in \cite{Cam-N-S}).  A natural generalization of the ``nodal separators'' of Mattei and Mar\'{\i}n \cite{Mar-M} is given by our definition of {\em nodal point} for a codimension one foliation in any dimension; these are the points where the foliation is locally given, in local coordinates $x_1,x_2,\ldots, x_n$, by $\omega=0$ where
$$
\omega=\sum_{i=1}^\tau \lambda_i\frac{dx_i}{x_i};\quad \lambda_i\in {\mathbb C^*},\; i=1,2,\ldots,\tau
$$
with $\lambda_i/\lambda_j\in {\mathbb R}$, for any $i,j$, and $\lambda_s/\lambda_j\in {\mathbb R}_{<0}$ for at least two indices
$s,j$. It is known that the singular locus $\mathrm{Sing}\,{\pi^*{\mathcal F}}$ is a union of nonsingular curves.  One such curve is {\em generically nodal} provided its generic point is a nodal point.  A {\em nodal component} $\mathcal N$ of the pair $\pi^*{\mathcal F}, E$ \/ is a connected component of the union of generically nodal curves
such that all the points in $\mathcal N$ are nodal points (and not only the generic points of the curves).

A key remark for the understanding of germs of foliations without an invariant germ of surface is that they must be {\em dicritical}. In a general way, we say that $\mathcal F$ is dicritical if there is a holomorphic map germ
$$
\phi:({\mathbb C}^2,0)\rightarrow ({\mathbb C}^3,0);\quad (x,y)\mapsto \phi(x,y)=(\phi_1(x,y),\phi_2(x,y),\phi_3(x,y))
$$
such that $\phi(\{y=0\})$ is invariant by $\mathcal F$ and the pullback $\phi^*{\mathcal F}$  coincides with the foliation $dx=0$ in $({\mathbb C}^2,0)$. In the paper \cite{Can-C} it is proved that any non dicritical foliation in $({\mathbb C}^3,0)$ has an invariant germ of analytic hypersurface, this is also true in any ambient dimension \cite{Can-M}. In fact, the arguments of \cite{Can-C} may be extended to the case where all compact components of the exceptional divisor are invariant; note that an irreducible component $D$ of  $E$ is compact if and only if $D\subset \pi^{-1}(0)$. Thus, if  $\mathcal F$ is without invariant surfaces, there is at least one compact component $D$ of $E$ that is generically transversal (dicritical component).

The main idea for Theorem \ref{teo:mainI} is that all the leaves of $\pi^*{\mathcal F}$ must intersect the union of compact dicritical components. At the intersection points we detect a germ of analytic curve contained in the leaf, which projects over the desired germ of analytic curve in $({\mathbb C}^3,0)$. The obstruction to having this property is the possible existence of nodal components, that could ``attract the leaves''.

The second result in this paper concerns the structure of the nodal components for a particular type of foliations that we call RICH-foliations. The idea is that we will be able to detect the possible existence of a nodal component $\mathcal N$ before doing the reduction of singularities, in the sense that $\mathcal N$ should project onto at least one of the curves $\Gamma\subset ({\mathbb C}^3,0)$ of the singular locus and the transversally generic behavior of $\Gamma$ is either dicritical or has a nodal separator in the sense of Mattei-Mar\'{\i}n.  To be precise, we prove the following result

\begin{theorem}\label{teo:mainII}
 Let $\mathcal F$ be a RICH-foliation in $({\mathbb C}^3,0)$. Assume that there is no germ of invariant analytic surface for $\mathcal F$. Then one of the two properties holds
\begin{itemize} \item[(i)] There is a neighborhood $U$ of the origin $0\in {\mathbb C}^3$ such that, for each leaf $L\subset U$ of $\mathcal F$ in $U$ there is an analytic curve $\gamma\subset L$ with $0\in \gamma$. \item[(ii)] There is an analytic curve $\Gamma$ contained in the singular locus $\mathrm{Sing}\,{\mathcal F}$ such that, $\mathcal F$ is generically dicritical or it has a nodal separator along $\Gamma$. \end{itemize} \end{theorem}

Let us explain the concepts appearing in Theorem \ref{teo:mainII}.
The term {\em RICH-foliation} stands for {\em Relatively Isolated Complex Hyperbolic Foliation}.
A  germ $\mathcal F$ of singular holomorphic foliation of codimension one in $({\mathbb C}^3,0)$ is a {\em RICH-foliation} if it is a CH-foliation and, furthermore, there is a reduction of singularities for $\mathcal F$
\begin{equation*}
\label{eq:redsing}
\pi: (M,\pi^{-1}(0))\rightarrow({\mathbb C}^3,0)
\end{equation*}
 where $\pi$ is a composition of blow-ups $\pi=\pi_1\circ\pi_2\circ\cdots\pi_N$
such that for any index $0\leq k\leq N-1$ the blow-up $\pi_{k+1}:M_{k+1}\rightarrow M_{k}$ satisfies
\begin{itemize}
\item The center $Y_{k}\subset M_{k}$ of $\pi_{k+1}$ is non singular, has normal crossings with the total exceptional divisor $E^{k}\subset M_{k}$ and is contained in the adapted singular locus $\mbox{Sing}({\mathcal F}_{k},E^{k})$, where ${\mathcal F}_{k}$ is the transform of  $\mathcal F$.
\item The intersection $Y_{k}\cap (\pi_1\circ\pi_2\circ \cdots \circ \pi_{k})^{-1}(0)$ is a single point.
\end{itemize}
(The {\em adapted singular locus}  $\mbox{Sing}({\mathcal F}_{k},E^{k})$ is the locus of points where ${\mathcal F}_k$ is singular or it does not have normal crossings with $E^k$, in particular $Y_{k}$ is invariant by  ${\mathcal F}_{k}$. For more details, see \cite{Can-C,Can}). This kind of reduction of singularities will be called a {\em RI-reduction of singularities of $\mathcal F$}.

The condition ``relatively isolated" is less restrictive than ``absolutely isolated''. It contains as examples the case of equireduction along a curve and the  foliations of the type $df=0$ where $f=0$ defines a germ of surface with absolutely isolated singularity. There are also examples without invariant surface, for instance the classical conic foliation given by Jouanolou \cite{Jou}.
Absolutely isolated singularities of vector fields have been studied in \cite{Cam-C-S}, whereas for codimension one foliations on $({\mathbb C}^3,0)$ the singular locus has codimension two unless we have a holomorphic first integral \cite{Mal}. Also, in the paper \cite{Can-C-S} the authors consider foliations desingularized essentially by punctual blow-ups, which is a condition stronger than being relatively isolated.

Following \cite{Mar-M}, we say that a germ of foliation ${\mathcal G}$ on $({\mathbb C}^2,0)$ contains a {\em nodal separator} if, in the reduction of singularities, there is a singularity analytically equivalent to $xdy-\lambda ydx=0$ were $\lambda$ is a non rational positive real number.
Now, take a germ of curve $\Gamma$ contained in the singular locus of a foliation $\mathcal F$ in $({\mathbb C}^3,0)$.
We say that $\mathcal F$ is {\em generically dicritical along $\Gamma$} if it is dicritical at a generic point of $\Gamma$.  We can  verify this fact at the equireduction points of $\Gamma$ (see \cite{Can}). If $\mathcal F$ is not generically dicritical along $\Gamma$, it is known \cite{Can} that the equireduction along $\Gamma$ is given by the (non-dicritical) reduction of singularities of the restriction $\mathcal G$ of $\mathcal F$ to a plane section transversal to $\Gamma$ at a generic point. In this case, we say that $\mathcal F$ {\em has a nodal separator along $\Gamma$} if this is true for such plane transversal sections $\mathcal G$.

Finally, the condition (ii) of Theorem \ref{teo:mainII} is equivalent to the fact that any nodal component intersects the dicritical components or it contains a non compact curve. To be precise, Theorem \ref{teo:mainII} is a consequence of Theorem \ref{teo:mainI} and the following result of structure for the nodal components

\begin{theorem}
\label{teo:mainIII}
 Let $\mathcal F$ be a RICH-foliation in $({\mathbb C}^3,0)$ and let $
 \pi: (M,\pi^{-1}(0))\rightarrow({\mathbb C}^3,0)
 $ be an RI-reduction of singularities
 for $\mathcal F$ with total exceptional divisor $E\subset M$. Any  compact nodal component $\mathcal N$ of $\pi^*{\mathcal F}, E$ intersects the union of the dicritical components of $E$.
\end{theorem}

It is an open question if the analogous of Theorem \ref{teo:mainIII} is true for CH-foliations.

In some sense the global alternative of Brunella may be interpreted as a property concerning the ``concentration-diffusion'' of the non-transcendency of the leaves of a foliation: either we concentrate the non-transcendency  in an algebraic leaf, or all the leaves are not completely transcendent in the sense that they are foliated by algebraic curves. In our local situation we have an analogous of this phenomenon based on the concept of ``end of a leaf''. In a forthcoming paper we will study the ends of the leaves for CH-foliations without invariant surface. All these ends will be ``semi-transcendental'' in the sense that, either they contain an analytic curve, or they are of a ``valuative type'' that admits bifurcation at all the accumulation points after blow-up. Moreover, the leaves in a neighborhood will have at least one end and in this sense we can reformulate a local version of Brunella's alternative by saying that, either we have an invariant germ of surface, or there is a neighborhood of the origin such that all the leaves are ``semi-transcendental''.\\

\noindent{\small {\em Acknowledgements:} We are very grateful to D. Cerveau and J.F. Mattei for the stimulating discussions we have had.}

\section{Nodal and Saddle Simple Complex Hyperbolic Points}

 We introduce here the  {\em simple complex hyperbolic points}, which are the higher-dimensional version of the simple singularities in the sense of Seidenberg (see  \cite{Can-C-D, Sei}) given by vector fields with two non-null eigenvalues.

Let $\mathcal F$ be a germ of singular holomorphic foliation of codimension one on $({\mathbb C}^n,0)$. We say that $\mathcal F$ has {\em dimensional type $\tau$ at the origin} if  there is a submersion
$$
\phi:({\mathbb C}^n,0)\rightarrow ({\mathbb C}^{\tau},0)
$$
and a codimension one foliation ${\mathcal G}$ on $({\mathbb C}^{\tau},0)$ such that ${\mathcal F}=\phi^*{\mathcal G}$ and moreover there is no such submersion $({\mathbb C}^n,0)\rightarrow ({\mathbb C}^{\tau-1},0)$.  In other words, there are local coordinates $x_1,x_2,\ldots,x_n$ at the origin $0\in {\mathbb C}^n$ and an integrable 1-form $\omega$ such that $\mathcal F$ is given by $\omega=0$, where $\omega$ can be written down as follows
$$
\omega=\sum_{i=1}^{\tau}a_i(x_1,x_2,\ldots,x_{\tau})dx_i
$$
and $\tau$ is the minimum integer with this property. We have that $\tau=1$ if and only if $\mathcal F$ is non-singular. Note also that
if there are $k$ germs of non-singular vector fields $\xi_1,\xi_2,\ldots,\xi_k$ tangent to $\mathcal F$ such that $\xi_1(0),\xi_2(0),\ldots,\xi_k(0)$ are $\mathbb C$-linearly independent tangent vectors, then $\tau\leq n-k$ and conversely.

\begin{definition}[\cite{Can,Can-C}]
 \label{def:simple}
 Let $\mathcal F$ be a germ of codimension one singular holomorphic foliation on $({\mathbb C}^n,0)$ of dimensional type $\tau$.
We say that $\mathcal F$ has a {\em simple complex hyperbolic point at the origin} if and only there are local coordinates $x_1,x_2,\ldots,x_n$ and a meromorphic integrable 1-form $\omega$ such that $\mathcal F$ is given by $\omega=0$ and $\omega$ can be written down as follows
\begin{equation*}
\label{eq:simple}
\omega=\sum_{i=1}^\tau (\lambda_i+b_i(x_1,x_2,\ldots,x_\tau))\frac{dx_i}{x_i},\; b_i\in {\mathbb C}\{x_1,x_2,\ldots,x_\tau\},\; b_i(0)=0
\end{equation*}
where
the
{\em residual vector  $\underline\lambda=(\lambda_i)_{i=1}^\tau\in {\mathbb C}^\tau$} satisfies the non-resonance property:\begin{quote}`` For any ${\mathbf m}=(m_i)_{i=1}^\tau\in {\mathbb Z}^\tau_{\geq 0}$ we have
$
\sum_{i=1}^\tau m_i\lambda_i\ne 0
$ if ${\mathbf m}\ne 0$''.
\end{quote}
\end{definition}
Let $E\subset ({\mathbb C}^n,0)$ be a normal crossings divisor. We decompose $E=E_{\mbox{\small inv}}\cup E_{\mbox{\small dic}}$, where $E_{\mbox{\small inv}}$ is the union of the irreducible components of $E$ invariant by $\mathcal F$ and $E_{\mbox{\small dic}}$ is the union of those that are generically transversal to $\mathcal F$ (dicritical components).
The origin is a {\em  simple complex hyperbolic point for ${\mathcal F}$ adapted to $E$} \/ if it is a simple complex hyperbolic point for $\mathcal F$ and  the coordinates in Definition \ref{def:simple} may be chosen in such a way that
$$(\prod_{i=1}^{\tau-1} x_i=0)\subset E_{\mbox{\small inv}}\subset(\prod_{i=1}^\tau x_i=0);\quad E_{\mbox{\small dic}}\subset(\prod_{i=\tau+1}^n x_i=0).$$
We adopt the following terminology:
\begin{itemize}
\item If
$
E_{\mbox{\small inv}}=(\prod_{i=1}^\tau x_i=0)
$,
 we have a {\em simple complex hyperbolic corner}.
\item If
$
E_{\mbox{\small inv}}=(\prod_{i=1}^{\tau-1} x_i=0)
$,
we  have a {\em simple complex hyperbolic trace point}.
\end{itemize}

\begin{notation} We denote {\em simple CH-point}, {\em simple CH-corner} or {\em simple CH-trace point} the above types of points.
\end{notation}

\begin{remark}
 \label{rek:singularlocus} When the origin is a simple CH-point as in Definition \ref{def:simple}, it is known
(\cite{Can,Can-C}) that the coordinate hyperplanes  $x_i=0$, where  $i=1,2,\ldots,\tau$ are the only invariant hypersurfaces of $\mathcal F$.
The singular locus $\mbox{Sing}{\mathcal F}$ is given by
$$ \mbox{Sing}{\mathcal F}=\cup_{1\leq i<j\leq \tau}(x_i=x_j=0). $$ \end{remark}

\begin{remark}[Formal normal forms]
\label{rk:formalnormalforms}
In the paper \cite{Can-C} it is shown that there are formal coordinates $\hat x_1,\hat x_2,\ldots,\hat x_n$ such that $\mathcal F$ is given at a CH-simple point by an integrable formal 1-form $\hat \omega$ of one of the following types:
\begin{enumerate}
\item $\hat\omega=\sum_{i=1}^\tau\lambda_i(d\hat x_i/\hat x_i)$, ($\underline{\lambda}$ non resonant).
\item $\hat\omega=\sum_{i=1}^\tau p_i(d\hat x_i/\hat x_i)+\hat \psi(\hat x_1^{p_1}\hat x_2^{p_2}\cdots\hat x_\tau^{p_\tau})\sum_{i=2}^\tau\mu_i(d\hat x_i/\hat x_i)$, where $\hat \psi(0)=0$.
\end{enumerate}
\end{remark}
\begin{definition} We say that a vector $\underline \lambda=(\lambda_1,\lambda_2,\ldots,\lambda_{\tau})\in {{\mathbb C}^*}^\tau$ is of {\em saddle type} if it is of one of the following types
 \begin{itemize}
 \item {\em Complex-saddle case}: There are two indices $i,j$ such that $\lambda_i/\lambda_j\notin {\mathbb R}$.
 \item {\em Real-saddle case}: $\lambda_i/\lambda_j\in {\mathbb R}_{>0}$, for any $i,j$.
 \end{itemize}
 Otherwise we say that $\underline\lambda$ is of {\em nodal type}, that is, $\lambda_i/\lambda_j\in {\mathbb R}$, for any $i,j$ and there are two indices $s,j$ such that $\lambda_s/\lambda_j\in {\mathbb R}_{<0}$.
\end{definition}
\begin{definition}
 \label{def:nodalsingularities} Let $\mathcal F$ be a germ of codimension one foliation in  $({\mathbb C}^n,0)$ of dimensional type $\tau$ having a simple CH-point at the origin.
The origin is of {\em saddle type (complex or real saddle)}, respectively  of {\em nodal type} if the residual vector
$
\underline \lambda
$
is so.
\end{definition}
\begin{remark}
\label{rk:linealizacionnodal} By a result of  Cerveau-Lins Neto \cite{Cer-LN} (see also \cite{Cer-M}), we know that
nodal singularities may be normalized in a convergent way. That is, if the residual vector
 is of nodal type, there are local coordinates $x_1,x_2,\ldots,x_n$ around the origin such that $\mathcal F$ is given by $\omega=0$ where
 \begin{equation}
 \label{eq:nodalequation}
 \omega=\sum_{i=1}^k{r_i}\frac{d x_i}{x_i}- \sum_{i=k+1}^\tau{r_i}\frac{d x_i}{x_i}; \; r_i\in {\mathbb R}_{>0},\, 1\leq k<\tau.
 \end{equation}
Note that the multi-valuated function $x_1^{r_1}x_2^{r_2}\cdots x_k^{r_k}/x_{k+1}^{r_{k+1}}x_{k+2}^{r_{k+2}}\cdots x_\tau^{r_\tau}$ is a first integral of the foliation.
\end{remark}
 \begin{remark} Complex-saddle singularities may also be normalized in a convergent way as a consequence of the results in \cite{Cer-LN} and thus they are expressed in convergent coordinates as $\omega=0$ where
\begin{equation}
\label{eq:logconvergente}
 \omega=\sum_{i=1}^\tau{\lambda_i}\frac{d x_i}{x_i},
  \end{equation}
  and there are two indices $i,j$ such that $\lambda_i/\lambda_j\notin {\mathbb R}$. On the other hand, real-saddle singularities are not necessarily given by a 1-form expressed in convergent coordinates as in Equation \ref{eq:logconvergente}. This is due to two possible facts: the existence of ``small denominators'' or a formal normal for of type (2) as in Remark \ref{rk:formalnormalforms}.
 \end{remark}
 
 \section{Leaves Around a Saddle Point}
 In this section we give a description of the behavior of the leaves of a foliation $\mathcal F$ that has a simple CH-point at the origin of ${\mathbb C}^n$ of saddle type and which is of dimensional type $\tau$. More precisely, we are interested in the saturation by $\mathcal F$ of small transversal curves to the coordinate hyperplanes $x_i=0$, for $i=1,2,\ldots,\tau$.  In dimension two, computations of this nature may be found in \cite{Mar-M}.

 Up to taking appropriate coordinates and a small enough neighborhood of the origin, we suppose that $\mathcal F$ is defined in the polydisc  $U={\mathbb D}_\delta^n$ around the origin by the integrable 1-form
$$
\omega=\sum_{i=1}^\tau (\lambda_i+b_i(x_1,x_2,\ldots,x_\tau))\frac{dx_i}{x_i},\; b_i(0)=0, \vert b_i\vert < \vert\lambda_i\vert.
$$
Denote $E_\ell=(x_\ell=0)$, $E=(\prod_{i=1}^\tau x_i=0)$ and $E_{\ell}^\circ=E_\ell\setminus \cup_{i=1, i\ne \ell}^\tau E_i$. Consider a small nonsingular curve $\Delta$ transversal to $E$ at a point  $Q\in E^\circ_\ell$. We are interested in the saturation $\mbox{Sat}_{{\mathcal F},U}\Delta$ of $\Delta$ by the leaves of $\mathcal F$ in $U$. More precisely, this section is devoted to giving a proof of the following result
\begin{proposition}
\label{prop:sillas} If the origin of $({\mathbb C}^n,0)$ is a simple CH-point of saddle type for $\mathcal F$, then
$
(\mbox{Sat}_{{\mathcal F},U}\Delta )\cup E
$ is a neighborhood of the origin, where $\Delta$ is a small curve transversal to $E$.
\end{proposition}

\begin{remark} The situation in the case of a nodal type point is different from the one described in Proposition \ref{prop:sillas}. Let $\mathcal F$ be given by $\omega=0$ as in Equation \ref{eq:nodalequation}. For any positive constant $C\in {\mathbb R}_{>0}$, the sets
$$
S_C=\{(z_1,z_2,\ldots,z_n); \frac{ \vert z_1\vert^{r_1} \vert z_2\vert^{r_2} \cdots  \vert z_k\vert^{r_k} }{ \vert z_{k+1}\vert^{r_{k+1}} \vert z_{k+2}\vert^{r_{k+2}} \cdots  \vert z_{\tau}\vert^{r_\tau} }=C\}
$$
are invariant sets for $\mathcal F$. If we take $\ell\leq k$ the curve $\Delta$ cuts only the sets $S_C$ with $0<C<\epsilon$, for some $\epsilon$. Noting that $S_{\epsilon+1}$ is adherent to the origin, we see that
$
\mbox{Sat}_{{\mathcal F},U}\Delta \cup E
$ is not a neighborhood of the origin. See \cite{Mar-M} for a description of this situation in dimension two.
\end{remark}
Let us make another remark for nodal singularities of dimensional type three
\begin{remark}
\label{rk:transicionnodal} Consider the foliation $\mathcal F$ defined in ${\mathbb D}_1^3\subset {\mathbb C}^3$ by $\omega=0$ with
$$
\omega=\frac{dx}{x}-\lambda\frac{dy}{y}-\mu\frac{dz}{z},
$$
where $\lambda,\mu\in {\mathbb R}_{>0}$. Up to a local coordinate change, all the nodal singularities of dimensional type three are of this form. Note that there are exactly two curves
$
x=y=0$   and $x=z=0
$
formed by nodal points and the curve $y=z=0$ has a generic point that is of real-saddle type. Moreover, by
a direct integration the saturation of a small neighborhood of any point in $(x=y=0)\cup (x=z=0)$ is a neighborhood of $xyz=0$.
\end{remark}

Proposition \ref{prop:sillas} comes by induction on the dimension and standard computations of holonomy as in \cite{Mar-M} or in \cite{Lor}. Before starting the proof, let us precise the notations. For
any $1\leq i,j\leq \tau$, write $\lambda_{ij}=\lambda_i/\lambda_j$ and $(\lambda_i+b_i)/(\lambda_j+b_j)=\lambda_{ij}+f_{ij}$. Then $\mathcal F$ is defined in $U$ by $\omega_j=0$ where
\begin{equation*}
 \omega_{j}=\frac{1}{\lambda_j+b_j}\omega= \sum_{i=1}^\tau(\lambda_{ij}+f_{ij})\frac{dx_i}{x_i}.
 \end{equation*}
Note that $\lambda_{jj}=1$, $f_{jj}=0$ and $f_{ij}(0)=0$.
By taking $\delta$ small enough, we assume that the following two properties are satisfied
\begin{enumerate}
\item[($\star$)] If we are in the complex-saddle case, then  $b_{i}=0$, for all $i$.
\item[($\star\star$)] In the real-saddle case, there is  $\rho<0$ such that $\vert f_{ij}\vert+\rho<\lambda_{ij}$ for all $i,j$.
\end{enumerate}
Take   $\ell\in\{1,2,\ldots,\tau\}$  and $\mu:\{1,2,\ldots,n\}\rightarrow {\mathbb D}_\delta$ such that $\mu_\ell=0$ and $\mu_i\ne 0$ for  $i\in \{1,2,\ldots,\tau\}\setminus \{\ell\}$.  Denote by $Q_{\mu}$ the point defined by $x_s(Q_{\mu})=\mu_s$.
Given  a radius $0<\epsilon\leq\delta$, we consider the curve $\Delta_\ell(\mu;\epsilon)$ over $Q_\mu$ defined by
$$
\Delta_\ell(\mu;\epsilon)=\{(x_1,x_2,\ldots,x_n); x_i=\mu_i, \mbox{ for } i\ne \ell, 0\leq \vert x_\ell\vert<\epsilon\}.$$
Now, we reformulate Proposition \ref{prop:sillas} as follows
\begin{quote}\em ``
$
\mbox{Sat}_{{\mathcal F},U}\Delta_\ell(\mu;\epsilon) \cup (\prod_{i=1}^\tau x_i=0)
$ is a neighborhood of the origin.''
\end{quote}
We start the proof of Proposition \ref{prop:sillas} with the case that the origin is of complex-saddle type. That is we have $\omega_j=\sum_{i=1}^\tau\lambda_{ij}dx_i/x_i$, where there is some $\lambda_{ij}\notin {\mathbb R}$.
\begin{lemma}
 \label{lemma:induccionnodal}
 Let $\underline\lambda=(\lambda_1,\lambda_2,\ldots,\lambda_\tau)\in {{\mathbb C}^*}^\tau$ be a vector of complex-saddle type and assume that $\tau\geq 3$. There are two indices $u,v$, $u\ne v$, such that the vectors
$$
\underline\lambda^u=(\lambda_1,\lambda_2,\ldots,\lambda_{u-1},\lambda_{u+1},\ldots,\lambda_\tau),\quad
\underline\lambda^v=(\lambda_1,\lambda_2,\ldots,\lambda_{v-1},\lambda_{v+1},\ldots,\lambda_\tau)
$$
are of complex-saddle type.
\end{lemma}
\begin{proof} Let $L_s$, $s=1,2,\ldots,k$ be the real rays (half real lines starting at the origin of $\mathbb C$) that contain all the $\lambda_s$, $s=1,2,\ldots, \tau$. We know that  $k\geq 2$. If $k\geq 3$, we take three distinct rays $L_1,L_2,L_3$ such that $L_3$ is not aligned (opposite ray to) with $L_1$ or $L_2$ and we consider $\lambda_{u}\in L_{1}$,  $\lambda_{v}\in L_{2}$; then $L_2$ and $L_3$ are rays for $\underline \lambda^u$ and $L_1$, $L_3$ are rays for $\underline \lambda^v$. If $k=2$, there are exactly two rays $L_1$, $L_2$ that are not opposite; one of these, say $L_1$ has at least two $\lambda_u,\lambda_v\in L_1$; now, $L_1$ and $L_2$ are still rays for for $\underline \lambda^u$ and for $\underline \lambda^v$.
\end{proof}

Now, Proposition \ref{prop:sillas} for complex-saddles is a consequence of Lemma \ref{lema:complexsaddle}:
\begin{lemma}
\label{lema:complexsaddle}
Assume that $\mathcal F$ has a simple CH-point at the origin of ${\mathbb C}^n$ of complex-saddle type.
We have
$
\mbox{Sat}_{{\mathcal F},U}\Delta_\ell(\mu;\epsilon) \cup (\prod_{i=1}^\tau x_i=0)={\mathbb D}_\delta^n
$.
\end{lemma}
\begin{proof} Given an index $u\in \{1,2,\ldots,\tau\}$ and $\alpha\in {\mathbb D}^*_\delta$, we consider the hyperplane
\begin{equation}
\label{eq:sigma}
 \Sigma_{\alpha}^u=\{x_u=\alpha\}\cap {\mathbb D}_\delta^n.
\end{equation}
The section ${\mathcal F}^u_{\alpha}$ of $\mathcal F$ by ${\Sigma_{\alpha}^u}$ has a CH-simple singularitiy at the points $Q_\nu$, where
$\nu:\{1,2,\ldots,n\}\rightarrow {\mathbb D}_\delta$ is such that $\nu_u=\alpha$ and $\nu_i=0$ for $i\in\{1,2,\ldots,\tau\}\setminus \{u\}$. Moreover ${\mathcal F}^u_{\alpha}$ is locally given at $Q_\nu$ by $\omega\vert_{\Sigma_\alpha^u}=0$ where
$$
\omega\vert_{\Sigma_\alpha^u}=
\sum_{i\in\{1,2,\ldots,\tau\}\setminus \{u\}}\lambda_i\frac{dy_i}{y_i};\; y_i=x_i\vert_{\Sigma_\alpha^u}.
$$
In particular, the residual vector is $\underline{\lambda}^u$  defined as in Lemma \ref{lemma:induccionnodal}.

 Let us do the induction step. Assume that the result is true for dimensional type $\tau'$ with  $2\leq\tau'<\tau$. We have to show that it is true for dimensional type  $\tau$.
 Choose indices $u,v$ as in Lemma \ref{lemma:induccionnodal}. We reduce first the problem to the case where $\ell\ne u,v$. Assume that $\ell=u$.  Applying induction hypothesis to the section ${\mathcal F}^v_{\mu_v}$, we deduce that
 $
 \Sigma_{\mu_v}^v
 $
 is contained in the saturation of
 $
\Delta_\ell(\mu;\epsilon)$.
 Take now $\ell'\in \{1,2,\ldots,\tau\}\setminus\{u,v\}$ and $\mu'$ defined by $\mu'_i=\mu_i$ for $i\in \{1,2,\ldots,n\}\setminus \{\ell',u\}$ and $\mu'_u\in {\mathbb D}^*_\delta$, $\mu'_{\ell'}=0$.
 We have that $\Delta_{\ell'}(\mu';\epsilon)\subset \Sigma_{\mu_v}^v$. The saturation of $\Delta_{\ell'}(\mu';\epsilon)$ is then contained in the saturation of $\Delta_\ell(\mu;\epsilon)$ and we are done.
 Thus we assume that $\ell\ne u,v$. Consider ${\mathcal F}_{\mu_u}^u$, applying induction hypothesis, we obtain that ${\Sigma}_{\mu_u}^u$ is contained in $\mbox{Sat}_{{\mathcal F},U}\Delta_\ell(\mu;\epsilon)$. Now, for any $\alpha\in{\mathbb D}^*_\delta$ we have that $$
 \Delta_\ell(\mu'(\alpha); \epsilon)\subset {\Sigma}_{\mu_u}^u,
 $$
 where $\mu'_i(\alpha)=\mu_i$ for $i\in\{1,2,\dots,\tau\}\setminus\{v\}$ and $\mu_v(\alpha)=\alpha$. Now, applying induction to ${\mathcal F}_\alpha^v$ we obtain that
 $
\mbox{Sat}_{{\mathcal F},U}\Delta_\ell(\mu'(\alpha);\epsilon)\supset {\Sigma_\alpha^v}
 $ and
taking the union over all the $\alpha$, we have
$$
\mbox{Sat}_{{\mathcal F},U}\Delta_\ell(\mu;\epsilon)\supset \mbox{Sat}_{{\mathcal F},U}{\Sigma}_{\mu_u}^u\supset
 \bigcup_{\alpha\in {\mathbb D}^*_\delta}\mbox{Sat}_{{\mathcal F},U}\Delta_\ell(\mu'(\alpha);\epsilon)\supset \bigcup_{\alpha\in {\mathbb D}^*_\delta}{\Sigma}_\alpha^v\supset {{\mathbb D}^*_\delta}^n.$$
 This ends the induction step.

 We end the proof by considering the case  $\tau=2$. It is enough to consider the case $\tau=n=2$ where
 \begin{equation*}
\omega=\frac{dx}{x}+\lambda\frac{dy}{y}; \; \lambda\in{\mathbb C}\setminus{\mathbb R}.
\end{equation*}
We assume also that $\ell=2$ and $\mu_1=\alpha\in {\mathbb D}^*_\delta$. Thus we take
\begin{equation}
\label{eq:dimdos}
\Delta(\alpha;\epsilon)=\{(x,y); x=\alpha, 0<\vert y\vert <\epsilon\}
\end{equation}
and we have to show that
$
\mbox{Sat}_{{\mathcal F},U}\Delta(\alpha;\epsilon)\supset {{\mathbb D}^*_\delta}^2
$.
Take another $\alpha'\in {\mathbb D}^*_\delta$. We can connect the points $(\alpha,0)$ and $(\alpha',0)$ by a path contained in ${\mathbb D}^*\times\{0\}$. By doing the holonomy along this path, we deduce that there is $0<\epsilon'<\delta$ such that $\Delta(\alpha';\epsilon')$ is contained in $
\mbox{Sat}_{{\mathcal F},U}\Delta(\alpha;\epsilon)$. Then it is enough to prove that
\begin{equation}
\label{eq:satcomplexsaddle}
\mbox{Sat}_{{\mathcal F},U}\Delta(\alpha;\epsilon)\supset \Delta(\alpha;\delta),
\end{equation}
since we would have
$$
\mbox{Sat}_{{\mathcal F},U}\Delta(\alpha;\epsilon)\supset \bigcup_{\alpha'\in {\mathbb D}^*_\delta} \mbox{Sat}_{{\mathcal F},U}\Delta(\alpha';\epsilon')\supset \bigcup_{\alpha'\in {\mathbb D}^*_\delta} \Delta(\alpha';\delta)={{\mathbb D}^*_\delta}^2.
$$
Let us show that (\ref{eq:satcomplexsaddle}) holds. To see this, we compute the holonomy at $(\alpha,0)$ with respect to the loop $\sigma(t)=(\alpha\exp(it),0)$ starting at a point $(\alpha, \beta)$. Let $\gamma(t)$ be the lifted path, where
$
\gamma(t)=(\alpha\exp(it),g(t)), \; t\in {\mathbb R}
$.
The condition $\omega(\gamma'(t))=0$ means that
$
g'(t)=({-i}/{\lambda})g(t)
$. Then $g(t)$ is explicitly given by
$$
g(t)=\beta\exp(-it/\lambda)=\beta\exp((-\mbox{Im}(\lambda)/\vert\lambda\vert^2)t)\exp(-i\mbox{Re}(\lambda)t/\vert\lambda\vert^2).
$$
Since $\mbox{Im}(\lambda)\ne 0$ we have a contraction or an expansion. If it is a contraction, we take the positive time to  reach ${\Delta}(\alpha;\epsilon)$ from ${\Delta}(\alpha;\delta)$; if it is an expansion, taking the negative time we have a contraction.
\end{proof}
We give now a proof of Proposition \ref{prop:sillas} in the case we have a real-saddle point. We do it by means of several lemmas. Proposition \ref{prop:sillas} for real-saddles is a direct consequence of Lemma \ref{lema:finalrealsaddle}.

 \begin{lemma}
   \label{lema:integrabilidad} Take two elements $i,j\in\{1,2,\ldots,\tau\}$ and a point $Q\in \{x_i=x_j=0\}$, then $f_{ij}(Q)=0$.
 \end{lemma}
 \begin{proof} It is a direct consequence of the integrability property. For any $s=1,2,\ldots,n$, the integrability property $\omega_{j}\wedge d\omega_{j}=0$ implies that
 $$
 0=A_s\left(x_i\frac{\partial A_j}{x_i}-x_j\frac{\partial A_i}{x_j}\right)-
 A_i\left(x_j\frac{\partial A_s}{x_j}-x_s\frac{\partial A_j}{x_s}\right)+
 A_j\left(x_s\frac{\partial A_i}{x_s}-x_i\frac{\partial A_s}{x_i}\right)
 $$
 where $A_s=(\lambda_{sj}+f_{sj})$. Let $\overline{f_{ij}}$ be the restriction of $f_{ij}$ to $x_i=x_j=0$. By applying the integrability property along $x_i=x_j=0$, we find that
 $$
 \frac{\partial \overline{f_{ij}}}{\partial x_s}=0,\quad s\ne i,j.
 $$
 In particular the value $f_{ij}(Q)$ does not depend on the point $Q\in \{x_i=x_j=0\}$ and thus $f_{ij}(Q)=0$.
 \end{proof}
 The above Lemma    \ref{lema:integrabilidad} is useful for an argument of induction on $\tau$. More precisely, if we take
 $\Sigma=\Sigma^u_{\mu_u}$  as in Equation \ref{eq:sigma}, the  section ${\mathcal F}^u_{\mu_u}$ gives also a real-saddle at any point $Q$ in $x_u=\mu_u$, $x_i=0$, $i\in \{1,2,\ldots,\tau\}\setminus \{u\}$, given by
 $$
 \omega_{j}\vert_{\Sigma}= \frac{dy_j}{y_j}+\sum_{i\notin \{j,u\}}(\lambda_{ij}+f_{ij}\vert_\Sigma)\frac{dy_i}{y_i};\quad y_i=x_i\vert_\Sigma, \, y_j=x_j\vert_\Sigma;\;  i,j  \ne u.
 $$
Moreover, the property $(\star\star)$ is satisfied at $Q$, with same $\rho$ as for $\mathcal F$.
\begin{lemma}
\label{lema:realsaddledimdos}
 Assume that $\mathcal F$ is given at the origin of ${\mathbb C}^2$ by $\omega=0$, where $\omega$ is the  1-form
$
\omega={dx}/{x}+(\lambda+f){dy}/{y}
$
defined in $U={\mathbb D}_\delta^2$ and such that $\lambda\in {\mathbb R}_{>0}$ with $\vert f\vert+\rho< \lambda$.
Take $0<\epsilon<\delta$, a complex number  $\alpha\in {\mathbb D}_\delta\setminus {\mathbb D}_{\delta/2}$ and consider the curve
$
\Delta(\alpha;\epsilon)$ as in (\ref{eq:dimdos}). There is a constant $0<c<\delta$, depending only on $\epsilon$, $\lambda$, $\rho$ and $\delta$ such that
$
\mbox{Sat}_{{\mathcal F},U}\Delta(\alpha;\epsilon)\supset {{\mathbb D}^*_\delta}\times{\mathbb D}^*_c
$.
\end{lemma}
\begin{proof} Consider a point $(\alpha',\beta')\in ({\mathbb D}^*)^2$. We will show that if $\vert\beta'\vert<c$, then the holonomy over a path $\sigma(t)$ allows us to reach $\Delta(\epsilon;\alpha)$ from $(\alpha',\beta')$. We consider two particular cases, the general situation is a combination of both:

{\em First case: } $\alpha=\alpha'\exp(i\theta)$, where $0\leq \theta<2\pi$. Consider the path $\sigma(t)=(\alpha'\exp(it),0),\; 0\leq t<2\pi$ and let $\gamma(t)=(\alpha'\exp(it),g(t))$ be the lifted path such that $\gamma(0)=(\alpha',\beta')$. The condition $\omega(\gamma'(t))=0$ gives that
$
g'(t)={-ig(t)}/{(\lambda+f)}
$
and if $F(t)=g(t)\overline g(t)$, we have that $F'(t)=(-2\mbox{Im}(\lambda+f)/\vert\lambda+f\vert^2)F(t)$. Note that
$$
\frac{\vert2\mbox{Im}(\lambda+f)\vert}{\vert\lambda+f\vert^2}\leq \frac{2}{\vert\lambda+f\vert}< {2}/{\rho}.
$$
We conclude that $F(t)\leq \vert\beta'\vert^2 \exp(2t/\rho)$, for $0\leq t<2\pi$. Put
$\epsilon'=\epsilon\exp(-2\pi/\rho)$, then $(\alpha',\beta')\in \Delta(\alpha';{\epsilon'})$ is in the saturation of $\Delta(\alpha;\epsilon)$.

{\em Second case: }  $\alpha'=r\alpha$, for a positive real number $0<r<2$.
Take the path
$\phi(t)=(t\alpha,0),\; 0<t<2$ and
 denote $\psi(t)=(t\alpha,u(t))$ the lifted path, such that $\psi(r)=(\alpha',\beta')$. The condition  $\omega(\psi'(t))=0$ gives that $u'(t)=-u(t)/t(\lambda+f)$. Put $U(t)=u(t)\overline u(t)$. We have that
 $$
 U'(t)=-\frac{2\mbox{Re}(\lambda+f)}{\vert\lambda+f\vert^2}\frac{U(t)}{t}.
 $$
 Let us note that
 $$
- \frac{2\mbox{Re}(\lambda+f)}{\vert\lambda+f\vert^2}<-\frac{2\rho}{{\vert\lambda+f\vert^2}}<-2\rho/(\lambda+\rho)<0.
 $$
 Moreover, $1/t>1/2$ if $0<t<2$. Consider the case $r<1$, hence the function $U(t)$, $r\leq t\leq 1$ satisfies $U(t)<V(t)$ where $V(t)$ is a solution of $V'(t)=-\rho/(\lambda+\rho)V(t)$, with $U(r)=V(r)=\vert\beta'\vert^2$. That is, we have
 $$
 U(t)\leq \vert\beta'\vert^2 \exp(-\rho/(\lambda+\rho)(t-r)).
 $$
 We deduce that $U(1)\leq \vert\beta'\vert^2 \exp(-\rho/(\lambda+\rho)(1-r))\leq \vert\beta'\vert^2$. In this case, if we take $\epsilon'=\epsilon$, we obtain that $\Delta(\alpha';\epsilon')$ is in the saturation of $\Delta(\alpha';\epsilon')$.

 Consider the case $1<r<2$. Put $t(s)=r+s(1-r)$ and let us define $\tilde \phi(s)=\phi(t(s))$, $\tilde\psi(s)=\psi(t(s))$, $\tilde u(s)=u(t(s))$, $\tilde U(s)=U(t(s))$. We have
 $$
 \tilde U'(s)=(r-1)\frac{2\mbox{Re}(\lambda+f)}{\vert\lambda+f\vert^2}{\tilde U(s)}\; \mbox{ and }\;
 0<(r-1)\frac{2\mbox{Re}(\lambda+f)}{\vert\lambda+f\vert^2}\leq \frac{2(\lambda+\rho)}{\rho^2}.
 $$
Hence the function $\tilde U(s)$, $0\leq s\leq 1$ satisfies $\tilde U(s)<\tilde V(s)$ where $\tilde V(s)$ is a solution of $\tilde V'(t)=({2(\lambda+\rho)}/{\rho^2})\tilde V(t)$, with $\tilde U(0)=\tilde V(0)=\vert\beta'\vert^2$. That is, we have
 $$
 \tilde U(s)\leq \vert\beta'\vert^2 \exp(({2(\lambda+\rho)}/{\rho^2})s).
 $$
 We deduce that $\tilde U(1)\leq \vert\beta'\vert^2 \exp({2(\lambda+\rho)}/{\rho^2})$. If we take $\epsilon'=\epsilon\exp(-{2(\lambda+\rho)}/{\rho^2})$, we obtain that $\Delta(\alpha';\epsilon')$ is in the saturation of $\Delta(\alpha;\epsilon)$.

 Combining the two situations above, we can go in a holonomic way from $\Delta(\alpha';\epsilon')$ to $\Delta(\alpha'';\epsilon'')$, where  $\alpha''=\alpha\exp(i\theta)$ is such that $\alpha'=r\alpha''$ and from $\Delta(\alpha'';\epsilon'')$ to $\Delta(\alpha;\epsilon)$. If we take $\epsilon'=\epsilon''\exp(-2\pi/\rho)$ and
 $\epsilon''=\epsilon\exp(-{2(\lambda+\rho)}/{\rho^2})$, we obtain that  $\Delta(\alpha';\epsilon')$ is in the saturation of $\Delta(\alpha;\epsilon)$. Hence it is enough to select the constant
 $c=\epsilon \exp (-2(( \pi+1 )\rho+\lambda))/\rho^2)$.
\end{proof}
\begin{lemma}
 \label{lema:finalrealsaddle}
 There is a constant $0<c<\delta$ depending only on $\lambda$, $\rho$, $\epsilon$ and $\delta$ such that if $\delta/2<\vert\mu_i\vert<\delta$ for all $i\ne \ell$, we have $\mbox{Sat}_{{\mathcal F},U}\Delta_\ell(\mu;\epsilon)\supset \{(x_1,x_2,\ldots,x_n); 0<\vert x_i\vert <\delta, \mbox{ for } i\ne \ell \mbox{ and } 0<\vert x_\ell\vert <c\}$.
\end{lemma}
\begin{proof} We proceed by induction on $\tau$. The case $\tau=2$ is given by
Lemma \ref{lema:realsaddledimdos}. Assume that $\tau\geq 3$ and take two indices $u,v\ne \ell$. We consider $\Sigma^u_{\mu_u}$  as in Equation \ref{eq:sigma}. By induction hypothesis applied to the section ${\mathcal F}^u_{\mu_u}$, we have that
$$\mbox{Sat}_{{\mathcal F},U}\Delta_\ell(\mu;\epsilon)\supset \Sigma^u_{\mu_u}\cap \{0<\vert x_i\vert <\delta, \mbox{ for } i\ne \ell \mbox{ and } 0<\vert x_\ell\vert <\epsilon'\}$$
for a certain $0<\epsilon'<\delta$. Take an $\alpha\in {\mathbb D}^*_\delta$, we have that
$
{\Delta}_\ell(\mu';\epsilon')$ is contained in $\mbox{Sat}_{{\mathcal F},U}\Delta_\ell(\mu;\epsilon)$,  where $\mu'_i=\mu_i$, $i\ne v$ and $\mu'_v=\alpha$
. Now, we apply induction to ${\mathcal F}_\alpha^v$ to conclude that there is a constant $c$ such that
$$
\bigcup_{\alpha\in {\mathbb D}^*_\delta}
\Sigma^v_{\alpha}\cap \{0<\vert x_i\vert <\delta, \mbox{ for } i\ne \ell \mbox{ and } 0<\vert x_\ell\vert <c\}
$$
is contained in $\mbox{Sat}_{{\mathcal F},U}\Delta_\ell(\mu;\epsilon)$. We are done.
\end{proof}

\section{Complex Hyperbolic Foliations}
Here we define the class of {\em complex hyperbolic} foliations. It is the higher dimensional version of the generalized curves in dimension two (the reader may look at \cite{Moz-F} for more details in the non dicritical case and ambient dimension three).
\begin{definition}
 \label{def:CHfoliaiton}
 Let $\mathcal F$ be a germ of singular holomorphic foliation of codimension one on $({\mathbb C}^n,0)$. We say that $\mathcal F$ is a {\em complex hyperbolic foliation} (for short, a ``CH-foliation'' ) at the origin if for any map
$
\phi:({\mathbb C}^2,0)\rightarrow ({\mathbb C}^n,0)
$
generically transversal to $\mathcal F$,
the foliation ${\mathcal G}=\phi^*{\mathcal F}$
has no saddle-nodes in its reduction of singularities (it is a generalized curve in the sense of  \cite{Cam-N-S}).
\end{definition}
\begin{remark} By performing a two-dimensional reduction of singularities of $\mathcal G$ (see \cite{Sei}, or \cite{Can-C-D}) we see that  $\mathcal F$  is a CH-foliation on $({\mathbb C}^n,0)$ if and only if there is no map $
\phi:({\mathbb C}^2,0)\rightarrow ({\mathbb C}^n,0)
$ generically transversal to $\mathcal F$ such that ${\mathcal G}=\phi^*{\mathcal F}$ has a saddle-node at the origin.\end{remark}

\begin{lemma}
 \label{lema:CH local}
 Let $\mathcal F$ be a germ of singular holomorphic foliation of codimension one on $({\mathbb C}^n,0)$ having a simple CH-point at the origin. Then $\mathcal F$ is a CH-foliation.
\end{lemma}
\begin{proof}  If ${\mathcal F}$ is not a CH-foliation, there is a map $\phi:({\mathbb C}^2,0)\rightarrow ({\mathbb C}^n,0)$ such that ${\mathcal G}=\phi^*{\mathcal F}$ has a saddle-node at the origin. Let us show that this is not possible. By performing finitely many local blow-ups of $({\mathbb C}^2,0)$ we obtain $\pi:({\mathbb C}^2,0)\rightarrow({\mathbb C}^2,0)$ such that if $\psi=\phi\circ\pi$ we have that  $\psi^*{\mathcal F}$ has a saddle-node at the origin and
$$
\psi_i=U_i(z_1,z_2)z_1^{a_{i}}z_2^{b_{i}}, \quad a_i,b_i\in {\mathbb Z}_{\geq 0},\quad U_i(0)\ne 0,
$$
for 
$i=1,2,\ldots,n$, 
where we write
$
\psi(z_1,z_2)=\left(\psi_1(z_1,z_2), \psi_2(z_1,z_2),\ldots, \psi_n(z_1,z_2)\right)
$.
Recall that $\mathcal F$ is given by a 1-form of the type
$$
\omega=\sum_{i=1}^\tau (\lambda_i+f_i(x_1,x_2,\ldots,x_\tau))\frac{dx_i}{x_i},\; f_i\in {\mathbb C}\{x_1,x_2,\ldots,x_\tau\},\; f_i(0)=0.
$$
Put $\tilde f_i=f_i\circ\psi$.
Then we have
\begin{eqnarray*}
\psi^*\omega&=&\sum_{i=1}^{\tau}(\lambda_i+\tilde f_i)\frac{d\psi_i}{\psi_i}=\sum_{i=1}^{\tau}(\lambda_i+\tilde f_i)\{a_i\frac{dz_1}{z_1}+b_i\frac{dz_2}{z_2}+\frac{dU_i}{U_i}\}=\\
&=&
\sum_{i=1}^{\tau}a_i(\lambda_i+\tilde f_i) \frac{dz_1}{z_1}+
\sum_{i=1}^{\tau}b_i(\lambda_i+\tilde f_i) \frac{dz_2}{z_2}+
\sum_{i=1}^{\tau}(\lambda_i+\tilde f_i) \frac{dU_i}{U_i}.
\end{eqnarray*}
The foliation $\psi^*{\mathcal F}$ is then given by a vector field with eigenvalues $\beta=\sum_{i=1}^{\tau}b_i\lambda_i $ and $\alpha=-\sum_{i=1}^{\tau}a_i\lambda_i$. Since we are supposing it is a saddle-node, up to a reordering, we have $\alpha=0$, $\beta\ne 0$. By the non resonance of the residual vector, we know that $a_i=0$ for all $i=1,2,\ldots,\tau$.  We obtain that $\psi^*{\mathcal F}$ is non singular at the origin, given by the non-singular 1-form $z_2\psi^*\omega$. This is the desired contradiction.
\end{proof}
\begin{proposition}
  \label{prop:redsingrichfoliations}
  Let $\mathcal F$ be a germ of singular holomorphic foliation of codimension one on $({\mathbb C}^n,0)$. Assume that there is a composition of blowing-ups with nonsingular centers $\pi:M\rightarrow ({\mathbb C}^n,0)$  such that each $p\in M$ is a CH-simple point for  $\pi^*{\mathcal F}$. Then $\mathcal F$ is a CH-foliation.
\end{proposition}
\begin{proof}  Consider a map $\phi:({\mathbb C}^2,0)\rightarrow  ({\mathbb C}^n,0)$ generically transversal to $\mathcal F$. By the universal property of the blow-up, there is a map
$
\sigma: {\Delta}\rightarrow ({\mathbb C}^2,0)
$
composition of a sequence of blow-ups that lifts $\phi$ to $\pi$. That is, there is $\psi: {\Delta}\rightarrow M$ such that $\pi\circ\psi=\phi\circ\sigma$. By Lemma \ref{lema:CH local}, the foliation
$\tilde {\mathcal G}=\psi^*\pi^*{\mathcal F}=\sigma^*\phi^*{\mathcal F}$
is a generalized curve at the points of $\Delta$.  ``A fortiori" ${\mathcal G}=\phi^*{\mathcal F}$ is a generalized curve.
\end{proof}

As a direct consequence of Proposition \ref{prop:redsingrichfoliations},  a RICH-foliation is also a CH-foliation.

It is not excluded for a CH-foliation to be dicritical. For instance, the foliation in $({\mathbb C}^3,0)$ given by the integrable 1-form
$$
\omega= (y^{m+1}-zx^m)dx+(z^{m+1}-xy^m)dy+(x^{m+1}-yz^m)dz
$$
is a dicritical CH-foliation. This foliation has no invariant surface \cite{Jou}.

Another example of CH-foliations is provided by the logarithmic foliations given by a 1-form $\omega$ of the type
$
\omega=\sum_{i=1}^k\lambda_i{df_i}/{f_i}$, $\lambda_i\in {\mathbb C}, i=1,2,\ldots,k
$, which correspond to the ``levels'' of the multivaluated function $f_1^{\lambda_1}f_2^{\lambda_2}\cdots f_k^{\lambda_k}$.

\begin{remark} In ambient dimension three, there is a reduction of singularities for any germ of codimension one foliation \cite{Can}. A reduction of singularities $\pi:M\rightarrow ({\mathbb C}^3,0)$ of $\mathcal F$ is {\em complex hyperbolic} if all the points of $M$ are simple CH-points for $\pi^*{\mathcal F}$. By Proposition \ref{prop:redsingrichfoliations} we see that if $\mathcal F$ has a complex hyperbolic reduction of singularities then $\mathcal F$ is a CH-foliation and thus all the reduction of singularities of $\mathcal F$ are also complex hyperbolic. This property has been taken as a definition in \cite{Moz-F}, where the authors consider the so-called {\em generalized surfaces} that are the non-dicritical CH-foliations in ambient dimension three. In this situation they prove that the reduction of singularities of the invariant surfaces automatically gives the reduction of singularities of the generalized surface. Next we state a result of this type in any ambient dimension that can be proved as in the three dimensional case.
\end{remark}
\begin{proposition}
  \label{prop:richfoliations}
  Let $\mathcal F$ be a germ of non-dicritical CH-foliation on $({\mathbb C}^n,0)$ of dimensional type $n$. Assume that the  invariant hypersurfaces of $\mathcal F$ are exactly the coordinate hyperplanes $\prod_{i=1}^nx_i=0$. Then the origin is a simple CH-point.
\end{proposition}
\begin{proof} (See also \cite{Moz-F}). We give a sketch of the proof. Assume that $\mathcal F$ is locally given by $\omega=\sum_{i=1}^nf_idx_i/x_i$ and put $\Omega=(\prod_{i=1}^nx_i\omega)$.
Take a general linear plane section $\phi: ({\mathbb C}^2,0)\rightarrow ({\mathbb C}^n,0)$. By the transversality theorem of Mattei-Moussu \cite{Mat-M} and taking into account that the foliation is non dicritical, we have that $\phi^*\Omega$ has isolated singularity. Moreover, $\phi^*{\mathcal F}$ is a non dicritical generalized curve. The only invariant curves of $\phi^*\Omega$ are the lines $\phi^{-1}(x_i=0)$, $i=1,2,\ldots,n$. In fact if there were another invariant curve $\Gamma$ for $\phi^*{\mathcal F}$, we could use the arguments in \cite{Can-C} and \cite{Can-M} to find an hypersurface for $\mathcal F$ different from the coordinate hyperplanes. Now, in dimension two, it is known (see \cite{Cam-N-S}) that a non dicritical generalized curve having exactly $n$ invariant lines has multiplicity equal to $n-1$. That is $\phi^*\Omega$ has multiplicity $n-1$. Thus $\Omega$ also has multiplicity $n-1$. Now, up to a reordering and to multiplying  $\omega$ by a unit, we have that 
$$
\omega=\frac{dx_1}{x_1}+\sum_{i=2}^nf_i\frac{dx_i}{x_i}.
$$
Let $\pi:M\rightarrow ({\mathbb C}^n,0)$ and $\sigma:N\rightarrow ({\mathbb C}^2,0)$ be the blowing-ups of the respective origins. Then $\phi$ lifts to $\tilde \phi: N\rightarrow M$. Because of the multiplicity $n-1$  of $\mathcal F$ and $\phi^*{\mathcal F}$ and the fact that both foliations are non dicritical, we see that $\tilde \phi$ is transversal to $\pi^*{\mathcal F}$, moreover, all the points in $N$ are simple CH-points for  $\sigma^*\phi^*{\mathcal F}=\tilde\phi^*\pi^*{\mathcal F}$, in particular they are not saddle nodes. A direct computation allows us to deduce from this that the order of each $f_i$ is equal to $0$, that is, we can write
$$
\omega=\frac{dx_1}{x_1}+\sum_{i=2}^n(\lambda_i+b_i)\frac{dx_i}{x_i};\quad b_i(0)=0,\,\lambda_i\ne 0,\, i=2,3,\ldots,n.
$$
The vector $\lambda=(1,\lambda_2,\lambda_3,\ldots,\lambda_n)$ is not resonant, otherwise we find a dicritical component by doing an appropriate sequence of blowing-ups
\end{proof}
\section{Pre-Simple CH-Corners}

 Let $\mathcal F$ be a germ of CH-foliation on $({\mathbb C}^n,0)$  and let us consider a normal crossings divisor $E\subset ({\mathbb C}^n,0)$. In this section we prove that blowing-up pre-simple CH-corners produces only adapted singularities that are pre-simple CH-corners. Moreover this is a characteristic property, that is, if we blow-up a non pre-simple corner, we will also find  a singular point in the transformed space which is not a pre-simple CH-corner. It is important to remark that pre-simple CH-corners may be dicritical.
 Let us precise the statements and definitions.

 Let us recall the definition of the {\em adapted singular locus}
 $\mbox{Sing}({\mathcal F},E)$ (see \cite{Can,Can-C}). We say that $\mathcal F$ and $E$ {\em have normal crossings at $p$} if $\mathcal F$ is non singular at $p$ and $H\cup E$ defines locally a normal crossings divisor, where $H$ is the invariant hypersurface of $\mathcal F$ through $p$. Then  $
 \mbox{Sing}({\mathcal F},E)$ is the set of points $p$ such that ${\mathcal F}$ and $E$ do not have normal crossings at $p$.

 \begin{definition}[\cite{Can,Can-C}]
 \label{def:presimplecorner}
 Let $\mathcal F$ be a germ of codimension one singular holomorphic foliation on $({\mathbb C}^n,0)$ of dimensional type $\tau$.
We say that the pair ${\mathcal F},E$ has a {\em pre-simple complex hyperbolic corner at the origin} if and only if there are local coordinates $x_1,x_2,\ldots,x_n$  such that $E_{\mbox{\small inv}}=(\prod_{i=1}^\tau x_i=0)$ and $\mathcal F$ is given by $\omega=0$ where
\begin{equation}
\label{eq:presimple}
\omega=\sum_{i=1}^\tau (\lambda_i+b_i(x_1,x_2,\ldots,x_\tau))\frac{dx_i}{x_i},\; b_i\in {\mathbb C}\{x_1,x_2,\ldots,x_\tau\},\; b_i(0)=0,
\end{equation}
 with $\prod_{i=1}^\tau\lambda_i\ne 0$.
 \end{definition}

 Let $Y\subset ({\mathbb C}^n,0)$ be a nonsingular subspace of codimension $\geq 2$ having normal crossings with $E$ and invariant by $\mathcal F$. Let us do the blowing-up
 $
 \pi: M\rightarrow ({\mathbb C}^n,0)
 $
 with center $Y$
 and consider the normal crossings divisor $E'=\pi^{-1}(E\cup Y)\subset M$. Denote ${\mathcal F}'=\pi^*{\mathcal F}$ the transformed foliation of $\mathcal F$ by $\pi$.

  \begin{proposition}
 \label{pro:chcornerstability}
 If the origin is a pre-simple CH-corner for  ${\mathcal F}, E$, then any point $p\in \pi^{-1}(0)\cap \mbox{Sing}({\mathcal F}',E')$ is a pre-simple CH-corner for ${\mathcal F}', E'$.
 \end{proposition}
 \begin{proof}  For simplicity, we assume that $E=E_{\mbox{\small inv}}$, the general case can be done with similar computations. Let $\tau$ be the dimensional type of $\mathcal F$ and choose local coordinates $x_1,x_2,\ldots,x_n$ as in Definition \ref{def:simple}, where $E=(\prod_{i=1}^\tau x_i=0)$ and
 $$
 Y=(x_1=x_2=\cdots=x_s=0, x_{\tau+1}=x_{\tau+2}=\cdots=x_{\tau+t}=0)
 $$
with $1\leq s\leq \tau$, $0\leq t\leq n-\tau$ and $s+t\geq 2$ (note that $s\geq 1$, otherwise $Y$ is not invariant). Up to a reordering of the indices, it is enough to verify the cases with local coordinates $x'_1,x'_2,\ldots,x'_n$ at $p\in \pi^{-1}(0)\cap E'_{\mbox{\small inv}}$ given by (1) or (2) as follows
 \begin{enumerate}
 \item There are scalars $\mu_2,\mu_3,\ldots,\mu_{s},\nu_{\tau+1},\nu_{\tau+2},\ldots,\nu_{t}$ such that
        \begin{enumerate}
    \item $x_1=x'_1$.
    \item $x_i=(x'_i+\mu_i)x'_1$, for $i\in \{2,3,\ldots,s\}$.
    \item $x_i=(x'_i+\nu_i)x'_1$, for $i\in \{\tau+1,\tau+2,\ldots,\tau+t\}$.
    \item $x_i=x'_i$ for $i\notin \{1,2,\ldots,s\}\cup \{\tau+1,\tau+2,\ldots,\tau+t\}$.
      \end{enumerate}
 \item There are scalars $\nu_{\tau+2},\nu_{\tau+3},\ldots,\nu_{\tau+t}$ such that
    \begin{enumerate}
    \item $x_i=x'_ix'_{\tau+1}$, for $i\in \{1,2,\ldots,s\}$.
    \item $x_{\tau+1}=x'_{\tau+1}$.
    \item $x_i=(x'_i+\nu_i)x'_{\tau+1}$, for $i\in \{\tau+2,\tau+3,\ldots,\tau+t\}$.
    \item $x_i=x'_i$ for $i\notin \{1,2,\ldots,s\}\cup \{\tau+1,\tau+2,\ldots,\tau+t\}$.
    \end{enumerate}
 \end{enumerate}
 The second case occurs only if $t\geq 1$.
 Put $r=1$ if we are in the first case and $r=\tau+1$ if we are in the second one. In both cases $\pi^{-1}(Y)$ is locally given at $p$ by $x'_r=0$ and the divisor $E'$ is  given by
 \begin{equation*}
 E'=\left\{\begin{array}{lc}
 (\prod_{i\in \{1\}\cup \{i\in \{2,\ldots,\tau\}; \mu_i=0\}} x_i=0), &\mbox{ first case }\\
  (\prod_{i\in \{1,2,\ldots,\tau,\tau+1\}} x_i=0), &\mbox{ second case }
  \end{array}
 \right.
 \end{equation*}
 We recall that the foliation $\mathcal F$ is locally given at the origin by $\omega$ as in Equation \ref{eq:presimple} of Definition \ref{def:presimplecorner}.

 {\em Case (1)}. The $1$-form $\omega$ is locally given at $p$ as
$$
\omega=  (\tilde \lambda_1+\tilde b_1)\frac{dx'_1}{x'_1}+ \sum_{i\in B}(\lambda_i+\tilde b_i)\frac{dx'_i}{x'_i}
+ \sum_{i\in C}\frac{\lambda_i+\tilde b_i}{\mu_i+x'_i} dx'_i,
$$
where $\tilde \lambda_1=\sum_{i=1}^s \lambda_i$, the germs $\tilde b_j\in {\mathbb C}\{x'_1,x'_2,\ldots,x'_\tau\}$ are in the ideal generated by $x'_1,x'_{s+1},x'_{s+2},\ldots,x'_\tau$ and
$$
B=\{ i\in \{2,3,\ldots,\tau\}; \tilde\mu_i=0\};\quad C=\{i\in \{2,3,\ldots,\tau\}; \mu_i\ne 0\}.
$$
If
$\tilde \lambda_1 \ne 0$, the nonsingular vector fields
$$
\xi_i=  \frac{\lambda_i+\tilde b_i}{\mu_i+x'_i} x'_1\frac{\partial}{\partial x'_1}- (\tilde \lambda_1+\tilde b_1)\frac{\partial }{\partial x'_i};\quad i\in C,
$$
trivialize the foliation and up to an appropriate coordinate change we may assume that ${\mathcal F}'$ is given by a form of the type
$$
\omega=  (\tilde \lambda_1+\tilde b_1)\frac{dx_1}{x_1}+ \sum_{i\in B}(\lambda_i+\tilde b_i)\frac{dx'_i}{x'_i}.
$$
Thus, the point $p$ is  a pre-simple CH-corner with dimensional type $\tau'=\tau-\sharp C$.

Assume now that $\tilde\lambda_1=0$, in particular $s\geq 2$. Then $x'_1$ divides $\tilde b_1$ or not.

If $x'_1$ divides $\tilde b_1$, with $f=\tilde b_1/x'_1$, the component $x'_1=0$ is dicritical and thus we have
$
E'_{\mbox{\small inv}}=(\prod_{i\in B}x_i=0).
$
In particular $B\ne \emptyset$.
Write
$$
\omega= f{dx_1}+ \sum_{i\in B}(\lambda_i+\tilde b_i)\frac{dx'_i}{x'_i}
+ \sum_{i\in C}\frac{\lambda_i+\tilde b_i}{\mu_i+x'_i} dx'_i.
$$
Take an index $j\in B\cup C$ and consider the non singular vector field $\xi$ where
$$
\xi= (\lambda_j+\tilde b_j) \frac{\partial}{\partial x'_1}- fx_j\frac{\partial }{\partial x_j},\mbox{ if } j\in B;\quad
\xi=  \frac{\lambda_j+\tilde b_j}{\tilde\mu_j+x'_j} \frac{\partial}{\partial x'_1}- f\frac{\partial }{\partial x_j},\mbox{ if } j\in C.
$$
Now $\xi$ trivializes the foliation and we can assume that $f$ is identically zero, that is
$$
\omega= \sum_{i\in B}(\lambda_i+\tilde b_i)\frac{dx'_i}{x'_i}
+ \sum_{i\in C}\frac{\lambda_i+\tilde b_i}{\tilde\mu_i+x'_i} dx'_i.
$$
If $B=\emptyset$, since $s\geq 2$ we have that $C\ne \emptyset$ and $\omega$ is non singular and has normal crossings with $E'=(x'_1=0)$ at $p$, then we are done, since $p\notin \mbox{Sing}({\mathcal F}',E')$. If $B\ne \emptyset$, take $i\in B$ and  consider the trivializing nonsingular vector fields
$$
\xi_j=  \frac{\lambda_j+\tilde b_j}{\tilde\mu_j+x'_j} x'_i\frac{\partial}{x'_i}- (\lambda_i+\tilde b_i)\frac{\partial }{x_j};\quad j\in C.
$$
Then, we get that $p$ is a pre-simple CH-corner with dimensional type $\tau'=\tau-1-\sharp C$.

Now, we assume that $\tilde \lambda_1=0$ and $x'_1$ does not divide $\tilde b_1$. We can write $\tilde b_1=x'_1f+g(x'_{s+1},x'_{s+2},\ldots,x'_\tau)$,
 where $g(0)=0$ and $g\ne 0$. Note in particular that $s<\tau$. Let us choose integers
 $
 \alpha_{s+1},\alpha_{s+2},\ldots,\alpha_\tau\in {\mathbb Z}_{\geq 1}
 $
 such that $\psi(v)\ne 0$ where $\psi(v)=g(v^{\alpha_{s+1}},v^{\alpha_{s+2}},\ldots,v^{\alpha_{\tau}})$ and $\lambda\ne 0$ where $\lambda=\sum_{i=s+1}^\tau\alpha_i\lambda_i$.
Let us consider the map $\delta:({\mathbb C}^2,0)\rightarrow (M,p)$ given by $x'_1=u$, $ x'_j=v^{\alpha_j}, j=s+1,s+2,\ldots,\tau$ and $x'_j=0$ otherwise. We have
$$
\delta^*\omega=(u\phi(u,v)+\psi(v))\frac{du}{u}+(\lambda+\rho(u,v))\frac{dv}{v};\quad \psi(0)=0=\rho(0).
$$
This is a saddle node in contradiction with the hypothesis that $\mathcal F$ is a CH-foliation.

{\em Case (2)}. We have
$$
\omega=\sum_{i=1}^\tau (\lambda_i+\tilde b_{i})\frac{dx'_i}{x'_i}  +(\tilde \lambda_{\tau+1}+\tilde b_{\tau+1})\frac{dx_{\tau+1}}{x_{\tau+1}}; \; \tilde \lambda_{\tau+1}=\sum_{i=1}^s \lambda_i,
$$
where the $\tilde b_i$ are in the ideal generated by $x'_{\tau+1}$ and $x'_j$, $j\in \{s+1,s+2,\ldots,\tau\}$. If $\tilde\lambda_{\tau+1}\ne 0$, we see that $p$ is a pre-simple CH-corner for ${\mathcal F}', E'$ of dimensional type $\tau'=\tau+1$. Assume that $\tilde\lambda_{\tau+1}=0$.  As before, if $x'_{\tau+1}$ does not divide $\tilde b_{\tau+1}$ we find a contradiction with the fact that ${\mathcal F}$ is a $CH$ foliation. If $x'_{\tau+1}$ divides $\tilde b_{\tau+1}$, with $\tilde b_{\tau+1}=x'_{\tau+1}f$, the created exceptional divisor $x'_{\tau+1}=0$ is dicritical, the foliation is given by
$$
\omega=\sum_{i=1}^\tau (\lambda_i+\tilde b_{i})\frac{dx'_i}{x'_i}  +f{dx'_{\tau+1}}
$$
and we can trivialize it to find a pre-simple CH-corner with $\tau'=\tau$.
 \end{proof}
\begin{lemma}
\label{lema: CHcornersprojective}
Let $\mathcal G$ be a codimension one singular foliation in the projective space ${\mathbb P}^n_{\mathbb C}$ and consider the normal crossings divisor $D\subset {\mathbb P}^n_{\mathbb C}$ given in homogeneous coordinates by
$
D=(\prod_{i=0}^eX_i=0)
$.
Assume that all the components of $D$ are invariant by $\mathcal G$ and any point in $\mbox{Sing}({\mathcal G}, D)$ is a pre-simple CH-corner for $\mathcal G$, $D$. Let $\mathcal G$ be defined by a non-null homogeneous integrable 1-form
$$
W=\sum_{i=0}^nA_i(X_0,X_1,\ldots,X_n)\frac{dX_i}{X_i};\quad \sum_{i=0}^nA_i=0,
$$
where the (non-null) coefficients $A_i$ are homogeneous polynomials of degree $r$ without common factor. Then $r=0$ and
$
A_{e+1}=A_{e+2}=\cdots=A_n=0
$. More precisely, $W$ has the form
$
W=\sum_{i=0}^e\lambda_i{dX_i}/{X_i}$ with $\sum_{i=0}^e\lambda_i=0$ and $\prod_{i=0}^e\lambda_i\ne 0$.\end{lemma}
\begin{proof} It is known that $\mbox{Sing}{\mathcal G}$ is a nonempty subset of ${\mathbb P}^n_{\mathbb C}$. Thus $D\ne\emptyset$ and more precisely $e\geq 2$, otherwise we find singular points that are not pre-simple CH-corners.  Let us show now that $X_i$ divides $A_i$ for each $i=e+1,e+2,\ldots,n$. If $X_{e+1}$ does not divide $A_{e+1}$, the hyperplane $X_{e+1}=0$ is invariant for $\mathcal F$; since $X_0=0$ is also invariant, any point
$$
P\in (X_0=X_{e+1}=0)\setminus (\prod_{i=1}^eX_i=0)
$$
is a singular point $P\in \mbox{Sing}{\mathcal G}$ that cannot be a pre-simple CH-corner for $\mathcal G$, $D$. Thus we have
$A_i=X_iB_i$, for $ i=e+1,e+2,\ldots,n$,
where either $A_i=0$ or $B_i$ is a homogeneous polynomial of degree $r-1$.

If $r\geq 1$,  the set
$
Z=(A_0=A_1=\cdots=A_n=0)=(A_1=A_2=\cdots=A_n=0)
$
is a non-empty Zariski-closed subset of ${\mathbb P}^n_{\mathbb C}$ and no point $P\in Z$ can be a pre-simple CH-corner. Thus $r=0$. This implies that the degree of $B_i$ is equal to $-1$ for $i=e+1,e+2,\ldots,n$, that is $A_{e+1}=A_{e+2}=\cdots=A_n=0$ and moreover $A_j=\lambda_j\in {\mathbb C}$ for $j=0,1,\ldots,e$. We have that $\lambda_j\ne 0$ for each $j=0,1,\ldots,e$ because $X_j=0$ is invariant by $\mathcal G$.
\end{proof}
 \begin{proposition}
 \label{pro:notchcorners}
 If all $p\in \pi^{-1}(0)\cap\mbox{Sing}({\mathcal F}',E')$ are pre-simple CH-corners for ${\mathcal F}', E'$, the origin is  also a pre-simple CH-corner for ${\mathcal F}, E$.
 \end{proposition}
 \begin{proof} Choose local coordinates $x_1,x_2,\ldots,x_n$  and subsets $A,B,C\subset \{1,2,\ldots,n\}$ such that $$E_{\mbox{\small inv}}=(\prod_{i\in A}x_i=0);\; Y=(x_i=0;i\in B);\; E_{\mbox{\small dic}}= (\prod_{i\in C}x_i=0).$$
 Then $\mathcal F$ is given by a meromorphic  1-form
    $
    \omega=\sum_{i\in A}a_i{dx_i}/{x_i}+\sum_{i\notin A}b_idx_i
    $
    where the coefficients do not have a common factor and  $x_i$ does not divide $a_i$, for $i\in A$. Let us put $a_i=x_ib_i$ for $i\notin A$ and $\omega=\sum_{i=1}^na_i{dx_i}/{x_i}$. We first show that at least one of the coefficients $a_i$, $i\in A$ is a unit. Let $r$ be the generic order $r=\nu_{Y}(a_i;i=1,2,\ldots,n)$ of the coefficients $a_i$ along $Y$. Let us write $a_i= A_{i,r}+A_{i,r+1}+\cdots$, the decomposition of $a_i$ in homogeneous components with respect to the variables $x_i,i\in B$, where one of the $A_{i,r}$  is not identically zero. We also have that
    $$\overline{a}_i=a_i\vert_{(x_i=0; i\notin B)}= \overline{A}_{i,r}+\overline{A}_{i,r+1}+\cdots$$
    is the decomposition in homogeneous components of $\overline{a}_i\in{\mathbb C}\{x_i;i\in B\}$. It is enough to show that $r=0$ and there is an index $i$ such that $ \overline{A}_{i,0}\ne 0$ (note that $i\in A$).

    Let us denote $f=\sum_{i\in B}a_i$. We decompose $f=F_r+F_{r+1}+\cdots$ as before.
    Consider an affine chart of the blow-up by taking $i_0\in B$ and coordinates $x'_i=x_i$ for $i\notin B\setminus \{i_0\}$, $x'_i=x_i/x_{i_0}$ for $i\in B\setminus \{i_0\}$. The transformed foliation in this chart is given by the 1-form
    $$
    \omega'=\frac{1}{x^r_{i_0}}\left(f\frac{dx_{i_0}}{x_{i_0}}+\sum_{i\ne i_0}a_i\frac{dx'_i}{x'_i}\right)=
    f'\frac{dx_{i_0}}{x_{i_0}}+\sum_{i\ne i_0}a'_i\frac{dx'_i}{x'_i},
    $$
where $f'=f/{x^r_{i_0}}$ and $a'_i=a_i/{x^r_{i_0}}$ for $i\ne i_0$. We
note that the blow-up is dicritical if and only if $x_{i_0}$ divides $f'$, this is equivalent to saying that $F_r=0$.

Let us consider points $p'\in \pi^{-1}(0)$ belonging to the selected affine chart; that is
$$
x'_i(p')=0;i\notin B\setminus\{i_0\},\; x'_i(p')=\mu_i; i\in B\setminus \{i_0\}.
$$

  {\em First case: the blow-up is non dicritical}.
Since $p'$ is a pre-simple CH-corner, we have that $f'(p')\ne 0$. Repeating this argument for other affine charts, we deduce that $r=0$ and $\overline F_0\ne 0$. This implies that some $\overline A_{i,0}\ne 0$ and hence $a_i$ is a unit.

{\em Second case: the blow-up is dicritical}. We have two possibilities, either $A\subset B$ or there is $j\in A\setminus B$. In the second case, we have $a'_j(p')\ne 0$ for all $p'\in \pi^{-1}(0)$, this implies that $r=0$ and $\overline{A}_{j,0}\ne 0$ and hence $a_j$ is a unit. Thus we suppose that $A\subset B$. By the hypothesis, there is some $i\ne i_0$ such that $a'_{i}(p')\ne 0$, but if $i\notin A$ we have $a'_i(p')=0$. In a more precise way, the restriction $
\omega\vert_{\pi^{-1}(0)}
$ is given in homogeneous coordinates by
$$
\omega\vert_{\pi^{-1}(0)}=\sum_{i\in A}\overline{A}_{i,r}\frac{dx'_i}{x'_i}
$$
and all the points in $\pi^{-1}(0)$ are pre-simple CH-corners for ${\mathcal F}\vert_{\pi^{-1}(0)}$, $E'_{\mbox{\small inv}}\cap \pi^{-1}(0)$. Now, we apply Lemma \ref{lema: CHcornersprojective} to deduce that $r=0$ and $\overline{A}_{i,0}\ne 0$ for all $i\in A$. In this case we deduce already that the origin is a pre-simple CH-corner.

Now, let us end the proof. We know that there is an index $s\in A$ such that $a_s$ is a unit. Up to dividing by this unit, we can write
$$
\omega=\frac{dx_s}{x_s}+\sum_{i\in A\setminus \{s\}}a_i\frac{dx_i}{x_i}+\sum_{j\notin A}b_jdx_j.
$$
We can trivialize the foliation by the tangent vector fields $\xi_j=b_jx_s\partial/\partial x_s-\partial/\partial x_j$ for $j\notin A$. This allows us to suppose that $b_j=0$ for all $j\notin A$. The integrability condition also gives in this situation that $a_i\in {\mathbb C}\{x_j; j\in A\}$. Now, it is enough to prove that $a_j(0)\ne 0$ for all $j\in A\setminus \{s\}$. Assume that $a_j(0)=0$, we blow-up the axis $x_s=x_j=0$ and we find a saddle node, contradiction with the fact that $\mathcal F$ is a CH-foliation.
\end{proof}

\section{Compact dicritical components and partial separatrices}
\label{sec:compactdicriticalcomponents}
In this section we extend to the dicritical case some features of the
 argument in \cite{Can-C} to find invariant germs of surfaces for a germ of foliation $\mathcal F$ in $({\mathbb C}^3,0)$. We are focusing the case of CH-foliations, although the results are of a wider scope.  For other extensions of this argument, the reader may see \cite{Reb-R}.

 Let us consider a germ of CH-foliation $\mathcal F$ in $({\mathbb C}^3,0)$ and  a reduction of singularities
 \begin{equation}
 \label{eq.redsing}
 \pi: (M,\pi^{-1}(0))\rightarrow ({\mathbb C}^3,0)
 \end{equation}
 of $\mathcal F$ which is a composition of blow-ups with invariant nonsingular centers, where the exceptional divisor $E$ has normal crossings and each point  $p\in \pi^{-1}(0)$ is a CH-simple point for $\pi^*{\mathcal F}$ adapted to $E$. The existence of such a reduction of singularities is guaranteed by the main result in \cite{Can}. Note that since $\mathcal F$ may be a dicritical foliation, the morphism $\pi$ is not necessarily obtained from any reduction of singularities of the invariant surfaces as in \cite{Moz-F};  it is even possible that there are no invariant surfaces.

The fiber $\pi^{-1}(0)$ is a connected closed analytic subset of $M$ whose irreducible components have dimension two or dimension one. The irreducible components of dimension two of the fiber coincide with the compact irreducible components of the exceptional divisor $E$. Each irreducible component $\Gamma$ of dimension one of $\pi^{-1}(0)$ is contained in at least one non-compact irreducible component of $E$ and never contained in a compact irreducible component of $E$ (otherwise it coincides with it).

Let us briefly recall the argument of construction of invariant surfaces in \cite{Can-C}. Take a point
$p\in \mbox{Sing}(\pi^*{\mathcal F},E)$; it can be a simple CH-corner or a simple CH-trace point and the dimensional type $\tau$ is $2$ or $3$.
Assume that it is a simple CH-trace point. Then there is a unique germ of invariant surface $(S_p,p)$ in $p$ different from the invariant components of $E$ through $p$. Moreover $(S_p,p)$ has normal crossings with $E$. In the case $\tau=2$, the adapted singular locus $\mbox{Sing}(\pi^*{\mathcal F},E)$ is a nonsingular curve contained in the unique invariant component $E_j$ of $E$ through $p$. More precisely  $$\mbox{Sing}(\pi^*{\mathcal F},E)=S_p\cap E_j$$ locally at $p$ and the foliation is analytically equivalent to the germ of $\pi^*{\mathcal F}$
 at $p$ in the points of $S_p\cap E_j$ close to $p$.
 In the case $\tau=3$, there are exactly two invariant components $E_i,E_j$ of $E$ through $p$ and the two lines $
S_p\cap E_i
$ and $
S_p\cap E_j
$
correspond locally at $p$ to the singular simple CH-trace points; all of them, except $p$ itself, are of dimensional type two. Thus, the
set
\begin{equation}
\label{eq:straza}
\mbox{STr}(\pi^*{\mathcal F},E)=\{p\in \mbox{Sing}(\pi^*{\mathcal F},E);  p \mbox{ is a simple CH-trace point } \}.
\end{equation}
is the union of curves of $\mbox{Sing}(\pi^*{\mathcal F},E)$ that are generically contained in only one irreducible component of $E$. We call these curves the {\em s-trace curves}, to be generalized in Section 8.

Let us note that $\mbox{STr}(\pi^*{\mathcal F},E)$ defines a germ of analytic set along $\mbox{STr}(\pi^*{\mathcal F},E)\cap \pi^{-1}(0)$,  hence the irreducible components of $\mbox{Sing}(\pi^*{\mathcal F},E)$ are either compact curves contained in $\pi^{-1}(0)$ or germs of curves.

Let $C$ be a connected component of $\mbox{STr}(\pi^*{\mathcal F},E)$; then $C\cap \pi^{-1}(0)$ is also connected, and it is either reduced to one point or it is a finite union of compact curves.
The germs of invariant surface $S_p$, for $p\in C$, can be glued together (see \cite{Can-C}) to obtain a unique invariant surface $S_C$ that is a germ along $C\cap \pi^{-1}(0)$. An important remark is that the immersion
$$
(S_C, C\cap \pi^{-1}(0))\subset (M_N,\pi^{-1}(0))
$$
is a closed immersion (that is $S_C\cap \pi^{-1}(0)=C\cap \pi^{-1}(0)$) if and only if $C$ does not intersect any compact dicritical component of $E$. If this is the case, we can produce an invariant surface $\pi(S_C)$ for ${\mathcal F}$ by properness of the morphism $\pi$. This is the main argument in \cite{Can-C}.

We can extend the same type of construction by considering also non singular trace points as follows. Define the set $\mbox{Inv}(\pi^{-1}(0))$ to be the union of the irreducible components of $\pi^{-1}(0)$ that are invariant by $\pi^*{\mathcal F}$ and consider the closed analytic set
\begin{equation*}
\label{eq:itraza}
\mbox{ITr}(\pi^*{\mathcal F},E)=\{p\in \mbox{Inv}(\pi^{-1}(0));  p \mbox{ is a simple CH-trace point for } \pi^*{\mathcal F},E\}.
\end{equation*}
(Compare with Equation \ref{eq:straza}). The irreducible components of $\mbox{ITr}(\pi^*{\mathcal F},E)$ are points or compact curves contained in the fiber, but not necessarily contained in the adapted singular locus.
\begin{lemma}Given a point $p\in \mbox{ITr}(\pi^*{\mathcal F},E)$ there is exactly one germ of invariant surface $S_p$ at $p$ not included in  $E$ and moreover $S_p$ has normal crossings with $E$.
\end{lemma}
\begin{proof}
If $p$ is a singular point we have seen this property before. If $p$ is a simple CH-trace point not in $\mbox{Sing}(\pi^*{\mathcal F},E)$,  there are no invariant components of $E$ at $p$ and $S_p$ is the only leaf through $p$.
\end{proof}

Take a connected component $\tilde C$ of  $\mbox{ITr}(\pi^*{\mathcal F},E)$. We can glue together the invariant surfaces $S_p$ to obtain a unique germ $(S_{\tilde C}, \tilde C)$ of invariant surface. Exactly as before, the immersion
$$
(S_{\tilde C}, \tilde C)\subset (M_N,\pi^{-1}(0))
$$
is a closed immersion if and only if $\tilde C$ does not intersect any compact dicritical component of the exceptional divisor $E$.

The above two constructions are related as follows. Given a connected component $C$ of $\mbox{STr}(\pi^*{\mathcal F},E)$, we have that
$
C\cap \mbox{ITr}(\pi^*{\mathcal F},E)
$
is nonempty and connected. In particular, the germ $(S_{C},C\cap\pi^{-1}(0))$ is contained in the germ $(S_{\tilde C}, \tilde C)$ where $\tilde C$ is the connected component of $\mbox{ITr}(\pi^*{\mathcal F},E)$ that contains $
C\cap \mbox{ITr}(\pi^*{\mathcal F},E)
$.
The inclusion of germs of surfaces
$
(S_{C},C\cap\pi^{-1}(0))\subset (S_{\tilde C}, \tilde C)
$
is not necessarily a closed immersion. Moreover, due to the possible existence of curves
 $\Gamma\subset \mbox{ITr}(\pi^*{\mathcal F},E)$
 whose points are nonsingular for $\pi^*{\mathcal F}$, it is possible to have two connected components $C_1$ and $C_2$ of $\mbox{STr}(\pi^*{\mathcal F},E)$ such that $\tilde C$ is a common connected component of $\mbox{ITr}(\pi^*{\mathcal F},E)$ that contains $
C_1\cap \mbox{ITr}(\pi^*{\mathcal F},E)
$ and $
C_2\cap \mbox{ITr}(\pi^*{\mathcal F},E)
$. Hence we can have two non closed inclusions of germs
$$
(S_{C_1},C_1\cap\pi^{-1}(0))\subset (\tilde S, \tilde C)\supset (S_{C_2},C_2\cap\pi^{-1}(0)).
$$
\begin{definition}
Given a connected component $\tilde C$ of  $\mbox{\rm ITr}(\pi^*{\mathcal F},E)$, {\em the partial separatrix over $\tilde C$} is 
the germ of invariant surface $( S_{\tilde C},\tilde C)$.
\end{definition}
 Now, let us give some results for the case that $\mathcal F$ has no germ of invariant surface.

 \begin{proposition}
\label{pro:invariantfiber}
Assume that $\mathcal F$ has no germ of invariant surface. Then we have
\begin{enumerate}
\item Any one dimensional irreducible component  of $\pi^{-1}(0)$ is invariant by $\pi^*{\mathcal F}$.
\item Any connected component $\tilde C$ of $\mbox{\rm ITr}(\pi^*{\mathcal F},E)$ intersects at least one compact dicritical component of $E$.
\item Any connected component $C$ of $\mbox{\rm STr}(\pi^*{\mathcal F},E)$ intersects at least one compact dicritical component of $E$.
\item There is at least one compact dicritical component of $E$.
\end{enumerate}
\end{proposition}
\begin{proof}
 1. Let $A$ be a one dimensional irreducible component of $\pi^{-1}(0)$. Assume that $A$ is not invariant by $\pi^*{\mathcal F}$.  At a generic point $p\in A$, we have that $\pi^*{\mathcal F}$ is transversal to $A$ and $\pi^{-1}(0)$ locally coincides with $A$. Then there is a germ $(\tilde S, p)$ of invariant surface for $\pi^*{\mathcal F}$  transversal to $\pi^{-1}(0)$ and such that $\tilde S\cap \pi^{-1}(0)= \{p\}$. Thus we have a closed inclusion of germs $(\tilde S, p)\subset (M,\pi^{-1}(0))$.
Since $\pi$ is a proper morphism, the image $S=\pi(\tilde S)$ is a germ of invariant surface for $\mathcal F$ at $0\in {\mathbb C}^3$. This is the desired contradiction.

2. If $\tilde C$ does not intersect any compact dicritical component of $E$, the partial separatrix $(S_{\tilde C}, \tilde C)$ is closed in $(M,\pi^{-1}(0))$ and thus $S=\pi(S_{\tilde C})$ is an invariant surface for $\mathcal F$.

3. Same argument as in (2).

4. Assume that there are no compact dicritical components to find a contradiction.
 By taking enough two dimensional sections ${\mathcal F}\vert_{\Delta}$ where $\Delta\subset ({\mathbb C}^3,0)$ is non singular and transversal to $\mathcal F$
in the sense of Mattei-Moussu \cite{Mat-M}, we find an invariant curve $\gamma$ that is not included in any center of the sequence of blowing-ups. Let $\gamma'\subset M$ be the strict transform of $\gamma$ and put $\{p\}=\gamma'\cap E$. We know that $p\in \pi^{-1}(0)$. Let us see that $p\in \mbox{\rm ITr}(\pi^*{\mathcal F},E)$. If $p$ is in an invariant component $D$ of $E$, we are done, since $\gamma'\not\subset E$ thus $p$ is singular and it cannot be a corner point (at the corner points the only invariant curves are contained in the divisor). If $p$ is not in an invariant component of $E$, it is a trace point and it belongs to a one dimensional component $\Gamma$ of the fiber, but $\Gamma$ is invariant and thus $p\in \mbox{\rm ITr}(\pi^*{\mathcal F},E)$. Now it is enough to consider the connected component $\tilde C$ of $\mbox{\rm ITr}(\pi^*{\mathcal F},E)$ that contains $p$ and apply (2).
\end{proof}

\section{Complex Hyperbolic Foliations Without Nodal Components}
\label{sec: nonodalcomponents}
We devote this section to giving a proof of Theorem \ref{teo:mainI}. Consider a CH-foliation $\mathcal F$ in $({\mathbb C}^3,0)$ and fix a reduction of singularities $\pi$ as in Equation \ref{eq.redsing}. Let us denote ${\mathcal F}'=\pi^*{\mathcal F}$ and let $E$ be the exceptional divisor of $\pi$. Since all the points of $M$ are CH-simple for ${\mathcal F}',E$,  the singular locus $\mbox{Sing}{\mathcal F}'$ is equal to the adapted singular locus $\mbox{Sing}({\mathcal F}',E)$ and it is a union of nonsingular connected curves
$$
\mbox{Sing}{\mathcal F}'=\Gamma_1\cup \Gamma_2\cup \cdots\cup \Gamma_s.
$$
\begin{definition}
\label{def:genericnodaltype}
An irreducible component $\Gamma_i$ of $\mbox{Sing}{\mathcal F}'$ is {\em of generic nodal type} if and only if $\Gamma_i$ contains a nodal point of dimensional type two.
 \end{definition}
 { Note that this is equivalent to saying that all the points of $\Gamma_i$ of dimensional type two are of nodal type}.
 \begin{definition}
 \label{def:nodalcomponents}
 The {\em generic nodal set\/} ${\rm GN}({\mathcal F}',E)$ is the union of the irreducible components of $\mbox{Sing}{\mathcal F}'$  of generic nodal type. A connected component $\mathcal N$ of $
{\rm GN}({\mathcal F}',E)
$ is a {\em nodal component} for ${\mathcal F}',E$ if and only if all the points of $\mathcal N$
are of nodal type. We say that $\mathcal F$ is {\em without nodal components} if there is a reduction of singularities $\pi$ such that   ${\mathcal F}',E$ is without nodal components.
\end{definition}

\begin{remark}
 \label{rk:nodalofdimthree}
 Let $p\in \mbox{Sing}{\mathcal F}'$ be of dimensional type three. We recall that $\mbox{Sing}{\mathcal F}'$ is locally given at $p$  by three curves $\Gamma_1,\Gamma_2,\Gamma_3$. One of these curves, say  $\Gamma_1$, is never of generic nodal type. Concerning the other ones, we have that $p$ is of nodal type if and only if both $\Gamma_2$ and $\Gamma_3$ are of generic nodal type (see Remark \ref{rk:linealizacionnodal}). In particular, a connected component $\mathcal N$ of
 $
{\rm GN}({\mathcal F}',E)
$
is a nodal component if and only if there are two irreducible components $\Gamma_2$ and $\Gamma_3$ of $\mathcal N$ through any given point $p\in {\mathcal N}$ of dimensional type three.
\end{remark}

 Before starting the proof of Theorem \ref{teo:mainI}, let us give some combinatorial results concerning the
 irreducible components of the exceptional divisor $E$ and the partial separatrices $(S_{\tilde C}, \tilde C)$, where $\tilde C$ runs over the connected components of $\mbox{ITr}({\mathcal F}',E)$.  The elements of the set ${\mathcal E}({\mathcal F}',E)$ of {\em exceptional components} for ${\mathcal F}',E$ are by definition the irreducible components of $E$ and the partial separatrices $(S_{\tilde C},\tilde C)$ (identified one to one with the connected components $\tilde C$ of $\mbox{ITr}({\mathcal F}',E)$). Given two exceptional components $A_1$ and $A_2$, the intersection $A_1\cap A_2$ is the corres\-ponding intersection as germs; it is either the empty set or a finite union of disjoint nonsingular curves. To be precise, we have the following types of exceptional components $A\in {\mathcal E}({\mathcal F}',E)$
 \begin{enumerate}
 \item $A$ is a compact irreducible component of  $E$, it can be dicritical or invariant.
 \item $A$ is a non compact irreducible component of $E$. In this case it is a germ over a finite union of compact curves $A\cap \pi^{-1}(0)$. It can be dicritical or invariant.
 \item $A$ is a partial separatrix. It is a germ $A=(S_{\tilde C}, \tilde C)$ over a finite union $\tilde C$  of curves that is a connected component of $\mbox{ITr}({\mathcal F}',E)$. By construction $A$ is invariant by ${\mathcal F}'$.
 \end{enumerate}
 \begin{definition} An exceptional component $A\in {\mathcal E}({\mathcal F}',E)$ is called {\em regular} if and only if
 it is invariant or a compact dicritical component of $E$. {\rm (The non-regular exceptional components are the non compact dicritical components of $E$)}.
 \end{definition}
 \begin{remark}
  \label{rk:connectedirregular}Assume that ${\mathcal F}$ has no germ of invariant surface. Each irreducible component $\Delta$ of the fiber $\pi^{-1}(0)$ is contained in at least one regular exceptional component $B_\Delta$. If $\Delta$ has dimension two, we are done, since $\Delta$ itself is a compact component of $E$. If $\Delta$ has dimension one, it is invariant by Proposition \ref{pro:invariantfiber} and thus $\Delta$ is contained in the partial separatrix $(S_{\tilde C},\tilde C)$, where $\tilde C$ is the connected component of $\mbox{ITr}({\mathcal F}',E)$ that contains $\Delta$. In particular, since the fiber $\pi^{-1}(0)$ is connected, given two irreducible components $\Delta_1$ and $\Delta_2$ of $\pi^{-1}(0)$, we can find a finite chain of regular exceptional components
  $
  B_0,B_1,\ldots,B_t
  $
  such that $\Delta_1\cap B_0\ne\emptyset\ne B_s\cap \Delta_2$, and $B_{i-1}\cap B_i\ne\emptyset$ for $i=1,2,\ldots,t$.
 \end{remark}
 \begin{definition} Two regular exceptional components $A_1$ and $A_2$ are {\em  nodally-free connected} if and only  if $A_1=A_2$ or there is a finite  chain of regular exceptional components $$A_1=B_0,B_1,\ldots, B_k=A_2$$ such that  $B_{i-1}\cap B_{i}$ contains a not generically nodal curve, for $i=1,2,\ldots,k$.
 \end{definition}
 \begin{lemma}
  \label{lema: connexionnodallibre}
  Assume that the pair ${\mathcal F}'$, $E$ is without nodal components and ${\mathcal F}$ has no germ of invariant surface. Any given pair $A_1,A_2$ of regular exceptional components is nodally-free connected.
 \end{lemma}
 \begin{proof} Let us first reduce the problem to the case that $A_1\cap A_2\ne \emptyset$. We know that the exceptional divisor $E$ is connected. Thus we can find a finite chain
 $$
 A_1=B'_0,B'_1,B_2,\ldots, B'_s,B'_{s+1}=A_2
 $$
 such  that $B'_i\cap B'_{i+1}\ne \emptyset$ for $i=0,1,\ldots,s$ and $B'_1,B'_2,\ldots,B'_s$ are irreducible components of $E$.
 Now, suppose that $B'_i$ is the last non compact dicritical component in the list.
 We know that $B'_i$ is a germ along the connected finite union of compact curves $B'_i\cap \pi^{-1}(0)$. Thus, there are  irreducible components $\Delta_1$ and $\Delta_2$ of the fiber such that
 $$
 B'_{i-1}\cap \Delta_1\ne\emptyset,
 \Delta_1\cap B'_i\ne\emptyset,
 B'_i\cap \Delta_2\ne\emptyset, \Delta_2\cap B'_{i+1}\ne\emptyset.
 $$
 We can substitute $B'_i$ by the sequence $
 B_0,B_1,\ldots,B_{t}
 $ of regular exceptional components
 given in
 Remark \ref{rk:connectedirregular}.
 Applying finite induction in this way, we can suppose that all the $B'_i$ are regular exceptional components. This reduces the problem to the case that $A_1\cap A_2\ne\emptyset$.

 Now, assume that $A_1\cap A_2$ intersect only at generically nodal curves (note that $A_1$ and $A_2$ are necessarily invariant by ${\mathcal F}'$).
  Take a point $p\in A_1\cap A_2\cap \pi^{-1}(0)$. The intersection $A_1\cap A_2$ locally at $p$ coincides with a generically nodal curve $\Gamma$. Since there are no nodal components, we can find generically nodal curves $\Gamma=\Gamma_0,\Gamma_1,\ldots,\Gamma_k$  and points $
 p=p_0, p_1,\ldots,p_{k}, p_{k+1}
 $
 such that
 $$
 p_0\in \Gamma_0, p_1\in \Gamma_0\cap\Gamma_1,p_2\in \Gamma_1\cap\Gamma_2,\ldots,p_k\in \Gamma_{k-1}\cap\Gamma_k, p_{k+1}\in \Gamma_k,
 $$
where the points $p_{i}$ are nodal points for $i=1,2,\ldots,k$
 and  $p_{k+1}$ is not a nodal point. We shall proceed by induction on this length $k$. If $k=0$, then $p_1\in \Gamma$ is not a nodal point. Moreover, the curve $\Gamma$ is locally given at $p_1$ by $A_1\cap A_2$. Since the dimensional type of $p_1$ is three, there is an invariant exceptional component $B$ transversal to $\Gamma$
  at $p_1$ such that the intersections $\Delta_1=A_1\cap B$ and $\Delta_2\cap B$ are curves not generically nodal. Then $A_1$ and $A_2$ are nodally-free connected through $B$.

  Assume now that $k\geq 1$. The curve $\Gamma_1$ is locally at $p_1$ the intersection of two invariant exceptional components $B_1$ and $B_2$. Moreover one of them, say $B_1$ is equal to $A_1$ or $A_2$, say that $B_1=A_1$; then, $B_2\cap A_2$ defines at $p_1$ a curve $\Delta$ that is not generically nodal (see Remark \ref{rk:nodalofdimthree}). We have that $A_2$ is nodally-free connected with $B_2$, by induction on $k$, we also have that $B_2$ is nodally-free connected with $B_1=A_1$ and we are done.
 \end{proof}

Let us suppose that the pair ${\mathcal F}', E$ is without nodal components and $\mathcal F$ has no germ of invariant surface. In order to give a proof of Theorem \ref{teo:mainI} we need to find a neighborhood $U$ of the origin and hence we start by representing our objects and morphisms in appropriate sets. We consider an open neighborhood $U_0$ of the origin of ${\mathbb C}^3$ such that the following properties hold
\begin{enumerate}
\item The foliation $\mathcal F$ is represented in $U_0$. We denote by $\tilde {\mathcal F}$ the corresponding foliation on $U_0$.
\item The morphism of germs $\pi: (M,\pi^{-1}(0))\rightarrow ({\mathbb C}^3,0)$ is represented in $U_0$ by
$$
\tilde \pi: \tilde M=\pi^{-1}(U_0)\rightarrow U_0.
$$
Moreover, we ask $\tilde \pi$ to be a composition of blowing-ups with connected nonsingular centers, in such a way that the centers of $\pi$ are the corresponding germs of subvarieties.
\item The total exceptional divisor $\tilde E$ of $\tilde \pi$ is a normal crossings divisor. Note of course that $E$ coincides with the germ $(\tilde E, \tilde\pi^{-1}(0))$.
\item The points in $\tilde M$ are CH-simple points for $\tilde {\mathcal F'}, \tilde E$, where $\tilde{\mathcal F'}=\tilde\pi^*\tilde{\mathcal F}$.
 \item The irreducible components of $\mbox{Sing}{\tilde F}'$ correspond one to one with the irreducible components of   $\mbox{Sing}{F}'$. The last ones are germs of curves or compact curves, the first ones are non singular closed curves that can be non compact or compact ones. In particular, the  {\em generic nodal set\/} ${\rm GN}(\tilde {\mathcal F}',\tilde E)$ is defined as the union of the generically nodal irreducible components of $\mbox{Sing}{\tilde F}'$ and it is a representant of the germ ${\rm GN}({\mathcal F}',E)$.
\end{enumerate}
For each connnected component $\tilde C$ of $\mbox{ITr}({\mathcal F}',E)$ we fix an open neighborhood $U_{\tilde C}$ of $\tilde C$ and a  representant $\tilde S_{\tilde C}$ of the partial separatrix $(S_{\tilde C}, \tilde C)$ such that $\tilde S_{\tilde C}\subset U_{\tilde C}$ is a connected non singular closed surface that has normal crossings with $\tilde E\cap U_{\tilde C}$. The {\em exceptional components of $\tilde {\mathcal F}, \tilde E$} are the irreducible components of $\tilde E$ and the closed surfaces $\tilde S_{\tilde C}$.

Let us consider the subset $H\subset \tilde M$ defined as the union of leaves $\tilde L$ of $\tilde {\mathcal F}'$ that intersect at least one compact dicritical component of $\tilde E$.

\begin{proposition}
 \label{prop:nonodalcomponents}   $H\cup \tilde E$ is a neighborhood of $\pi^{-1}(0)$ in $\tilde M$.
 \end{proposition}
 \begin{proof} By Remark \ref{rk:connectedirregular} each irreducible component $\Delta$ of the fiber $\pi^{-1}(0)$ is contained in at least one regular exceptional component $B_\Delta$. In this way, we can include $\pi^{-1}(0)$ in the union of the regular exceptional components. By Lemma \ref{lema: connexionnodallibre} each two regular exceptional components are nodally-free connected. Now, we proceed as follows:
 \begin{enumerate}
 \item We show that $H\cup \tilde E$ is a neighborhood of $B\setminus {\rm GN}(\tilde {\mathcal F}',\tilde E)$, for each regular exceptional component $B$.
 \item We show that $H\cup E^N$ is a neighborhood of ${\rm GN}(\tilde {\mathcal F}',\tilde E)$.
 \end{enumerate}
 Let us prove the first assertion. If $B$ is a compact dicritical component, we are done. Otherwise, by Lemma \ref{lema: connexionnodallibre} and by Proposition \ref{pro:invariantfiber}
there is a finite chain of regular exceptional components that connects $B$ with a compact dicritical component $B'$ through non-nodal curves. Now we do a holonomic transport from $B'$ to $B$ that allows us to cover with the leaves arriving to $B'$ the part of $B$ outside the generically nodal curves. To do this we invoke the behavior of the leaves at the regular points and at the simple points that are not nodal ones, described in Proposition \ref{prop:sillas}.

 In order to prove the second assertion, given a connected component $\mathcal N$ of ${\rm GN}(\tilde {\mathcal F}',\tilde E)$, we find a non-nodal point $p\in {\mathcal N}$. Then $H\cup \tilde E$ is a neighborhood of this point and by saturation along $\mathcal N$, applying Remark \ref{rk:transicionnodal}, we cover $\mathcal N$.
\end{proof}
Now, Proposition \ref{prop:nonodalcomponents} implies Theorem \ref{teo:mainI} as follows.
The set $V=\tilde\pi(H\cup \tilde E)$ is a  neighborhood of the origin in $U_0$ saturated by $\mathcal F$. Let $W\subset {\mathbb C}^3$ be an open neighborhood of the origin with $W\subset V$. The saturation $\mbox{Sat}_{\tilde{\mathcal F}}W$ is open and contained in $V$. We take $U=\mbox{Sat}_{\tilde {\mathcal F}}W$. Let $L$ be a leaf in $U$.
 Now $\tilde\pi^{-1}(L)\setminus \tilde E$ is connected and invariant by $\tilde {\mathcal F}$. Let $\tilde L$ be the leaf of $\tilde {\mathcal F}$ containing $\tilde\pi^{-1}(L)\setminus \tilde E$.
 \begin{lemma} $\tilde L\setminus \tilde E=\tilde\pi^{-1}(L)\setminus \tilde E$.
 \end{lemma}
 \begin{proof}  Take  points $p\in \tilde\pi^{-1}(L)\setminus \tilde E$ and  $q\in \tilde L\setminus \tilde E$. Since the two points are in $\tilde L$, there is a compact path $\delta(t)$, $\delta(0)=p$, $\delta(1)=q$ such that $\delta(t)\in \tilde L$ for $t\in[0,1]$. By a local study at the points of the dicritical components,  the set of the $t\in [0,1]$ such that
 $$
\delta(t)\in (\tilde\pi^{-1}(L)\setminus \tilde E)\cup \bigcup\{E^N_j;\; E^N_j\mbox{ is a dicritical component of } E^N\}
 $$
 is closed and open in $[0,1]$. Thus, $q=\delta(1)\in \tilde\pi^{-1}(L)$.
 \end{proof}

 Now,  since $\tilde L\subset H$, there is a compact dicritical component $ E_j$ of $E$ such that $\tilde L\cap E_j\ne \emptyset$. We find a germ of non singular analytic curve $\tilde \gamma\subset \tilde L$ transversal to $E_j$ in a point
 $p\in E_j$ with $e(E,p)=1$. The projection $\gamma=\pi(\tilde\gamma)$ is a germ of curve contained in $L$.

\section{Singular Locus of a RICH-Foliation}
\label{sec:singlocus}
In this section we describe some features of the singular locus of a RICH-foliation $\mathcal F$ at the intermediate steps of a fixed  RI-reduction of singularities \begin{equation}
\label{eq:sucesionreduccion}\pi:(M,\pi^{-1}(0))\rightarrow ({\mathbb C}^3,0).\end{equation}
We recall that $\pi$ is a composition $\pi=\pi_1\circ\pi_2\circ\cdots\circ\pi_N$ of blow-ups $\pi_k:M_k\rightarrow M_{k-1},\,k=1,2,\ldots,N$ such that for any $0\leq k\leq N-1$ we have
\begin{enumerate}
\item The center $Y_{k}\subset M_{k}$ of the blow-up $\pi_{k+1}$ is non singular, has normal crossings with the total exceptional divisor $E^{k}\subset M_{k}$ and it is contained in the adapted singular locus $\mbox{Sing}({\mathcal F}_{k},E^{k})$, where ${\mathcal F}_{k}$ is the transform of  $\mathcal F$ (in particular it is invariant by  ${\mathcal F}_{k}$).
\item
 \label{ri2}
 The intersection $Y_{k}\cap (\pi_1\circ\pi_2\circ\cdots\circ\pi_{k})^{-1}(0)$ is a single point.
\item All the points of $M_N=M$ are CH-simple for  $\pi^*{\mathcal F}, E$, where $E=E^N$ is the total exceptional divisor.
\end{enumerate}
Given $0\leq k\leq k'\leq N$ we denote $\pi_{kk}=\mbox{id}_{M_k}$ and
$
\pi_{k'k}=\pi_{k+1}\circ\pi_{k+2}\circ\cdots\pi_{k'}
$
 if $k<k'$. We take special notations for some particular cases $\rho_{k}=\pi_{Nk}$, $\sigma_k=\pi_{k0}$, $\pi=\sigma_k\circ \rho_k$, with
 $\rho_{k}:M=M_N\rightarrow M_k$ and $\sigma_k:M_k\rightarrow ({\mathbb C}^3,0)$.
 We decompose the exceptional divisor $E^k$ into irreducible components
$$
E^k=E^k_1\cup E^k_2\cup \cdots \cup  E^k_k
$$
where $E^k_j$ is the strict transform by $\pi_k$ of $E^{k-1}_j$ for $j<k$ and $E^k_k=\pi_k^{-1}(Y_{k-1})$. We write $E^k_{\mbox{\small inv}}\subset E^k$, respectively $E^k_{\mbox{\small dic}}\subset E^k$, the union of the irreducible components of $E^k$ invariant by ${\mathcal F}_k$, respectively the generically transversal (dicritical) components of $E^k$.
\begin{remark}
\label{rk:firstpropertiesS}
The $\pi_{k'k}$  are morphisms of germs $\pi_{k'k}:(M_{k'},\sigma_{k'}^{-1}(0))\rightarrow(M_k,\sigma_k^{-1}(0))$ around the compact subsets $\sigma_k^{-1}(0)\subset M_k$ and $\sigma_{k'}^{-1}(0)\subset M_{k'}$.
An irreducible component $E^k_i$ of the exceptional divisor $E^k$ is compact if and only $E^k_i\subset \sigma_{k}^{-1}(0)$ and this is equivalent to saying that $Y_{i-1}\subset \sigma_{i-1}^{-1}(0)$. In view of Property (\ref{ri2}) of the  reduction of singularities, this is also equivalent to saying that $Y_{i-1}$ is a single point.
Conversely, the irreducible component $E^k_i$ is non-compact if and only if  the center $Y_{i-1}$ is a germ of curve. Moreover, in this case $Y_{i-1}$ is a germ of curve not contained in $\sigma_{i-1}^{-1}(0)$, in particular, it projects by $\sigma_{i-1}$ onto a curve in $M_0=({\mathbb C}^3,0)$.
\end{remark}
\begin{remark}
\label{rk:points in curves}
Let $\Gamma\subset M_k$ be a curve contained in the adapted singular locus of ${\mathcal F}_k, E^k$. By Property (\ref{ri2}), we see that only finitely many points of $\Gamma\cap\sigma_k^{-1}(0)$ are not simple points for ${\mathcal F}_k, E^k$. In particular, if $\Gamma$ is a compact curve, that is $\Gamma\subset \sigma_k^{-1}(0)$, all points in $\Gamma$, except maybe finitely many, are simple points for  ${\mathcal F}_k, E^k$; moreover, up to eliminating finitely many other points of dimensional type three, the foliation has dimensional type two along $\Gamma$.
\end{remark}
\begin{remark}
\label{rk:gradodicritico}
If $E^k_j$ is a compact dicritical component of ${\mathcal F}_k$, there are only finitely many points in $\mbox{Sing}({\mathcal F}_k,E^k)\cap E^k_j$. That is, there is no curve  $\Gamma$ contained in $\mbox{Sing}({\mathcal F}_k,E^k)\cap E^k_j$. If such a curve exists, by  Remark \ref{rk:points in curves}
all points in $\Gamma$, except may be finitely many of them, are simple points for ${\mathcal F}_k,E^k$; now, if $q\in \Gamma\cap E^k_j$ is a simple point for ${\mathcal F}_k,E^k$, the dimensional type of ${\mathcal F}_k$ in $q$ is two and $\Gamma$ is transversal to $E^k_j$, contradiction with the fact that $\Gamma\subset E^k_j$. This property has the following consequence in terms of local equations. Take a point $q\in E^k_j$ and suppose that ${\mathcal F}_k$ is locally given at $q$ by $\omega=0$ where
\begin{equation}
\label{eq:restriccion}
\omega=a(x,y,z)dx+b(x,y,z)dy+c(x,y,z)dz;\quad E^k_j=(z=0)
\end{equation}
and $a,b,c$ do not have common factors in ${\mathbb C}\{x,y,z\}$. Then the restriction $\mathcal G$ of ${\mathcal F}_k$ to $E^k_j$ is locally given at $q$ by $\eta=a(x,y,0)dx+b(x,y,0)dy$ and moreover $a(x,y,0), b(x,y,0)$ do not have common factors in ${\mathbb C}\{x,y\}$.
\end{remark}

\begin{definition}
 A point $p\in M_k$ is {\em a trace point for} ${\mathcal F}_k, E^k$ if and only if it is not a pre-simple CH-corner for ${\mathcal F}_k, E^k$. If in addition we have that $p\in \mbox{Sing}({\mathcal F}_k,E^k)$, we say that $p$ is an {\em s-trace point}. An irreducible curve $\Gamma\subset \mbox{Sing}({\mathcal F}_k,E^k)$ is an {\em s-trace curve for } ${\mathcal F}_k, E^k$ if and only if all the points of $\Gamma$ are s-trace points.
 \end{definition}
 \begin{remark}
  \label{rk:closednesstrace}
  By the local description of pre-simple CH-corners, all the singular points around a pre-simple CH-corner are also pre-simple CH-corners.  More precisely, the set of s-trace points
 \begin{equation*}
 \label{eq_straza2}
  \mbox{\rm STr}({\mathcal F}_k,E^k)=\mbox{Sing}({\mathcal F}_k,E^k)\setminus \{\mbox{pre-simple CH-corners}\}
  \end{equation*}
  is a closed analytic subset of $
  \mbox{Sing}({\mathcal F}_k,E^k)
  $. Hence the s-trace curves are the irreducible components of dimension one of $
  \mbox{\rm STr}({\mathcal F}_k,E^k)
  $.
 \end{remark}
 \begin{notation}
Given a point $p\in M$ and a normal crossings divisor $D\subset M$ we denote by $e(D;p)$ the number of irreducible components of $D$ passing through $p$. In the same way, if $\Gamma\subset M$ is an irreducible curve, we denote by $e(D;\Gamma)$ the number of irreducible components $D_j$ of $D$ such that $\Gamma\subset D_j$.
\end{notation}
 \begin{remark}
  \label{rk:divisortrace}
  Let $\Gamma\subset \mbox{Sing}({\mathcal F}_k,E^k)$ be an irreducible curve. We have the following properties
   \begin{enumerate}
   \item The curve $\Gamma$ is not an s-trace curve if and only if all but finitely many points in $\Gamma$ are pre-simple CH-corners of dimensional type two. In this case $\Gamma$ is the intersection of two invariant components $\Gamma=E^k_i\cap E^k_j$. We also say that $\Gamma$ is {\em generically pre-simple CH-corner}.
   \item
   Assume that the generic point of $\Gamma$ is simple for ${\mathcal F}_k,E^k$; note that this is always the case if $\Gamma$ is a compact curve $\Gamma\subset \sigma_k^{-1}(0)$. Then $\Gamma$ is an s-trace curve if and only if $e(E^k_{\mbox{\small inv}},\Gamma)\leq 1$.
  \end{enumerate}
  Moreover, in the case that $\Gamma$ is a compact curve, we have exactly the following two possibilities
\begin{enumerate}
\item[a)]  $e(E^k,\Gamma)=e(E^k_{\mbox{\small inv}},\Gamma)=2$ and $\Gamma$ is generically simple CH-corner.
\item[b)] $e(E^k,\Gamma)=e(E^k_{\mbox{\small inv}},\Gamma)=1$ and $\Gamma$ is an s-trace curve.
\end{enumerate}\end{remark}

Let us do some considerations about the fiber $\sigma_k^{-1}(0)$. It decomposes into irreducible components
\begin{equation*}
\sigma_k^{-1}(0)=A_1^k\cup A_2^k\cup\cdots\cup A_{s_k}^k
\end{equation*}
where
each one is either a compact irreducible component of $E^k$ or a compact nonsingular curve. More precisely, $\sigma_k^{-1}(0)$ has normal crossings with $E^k$. If $A_i^k$ is an irreducible component of $\sigma_k^{-1}(0)$ of dimension one, the intersection of $A_i^k$ with a compact component of $E^k$ is at most one point and there is at least one non compact component $E^k_j$ such that $A_i^k\subset E^k_j$.
Denote by $\mbox{Inv}(\sigma_k^{-1}(0))\subset \sigma_k^{-1}(0)$ the union of the invariant irreducible components of $\sigma_k^{-1}(0)$.

Let us recall that if $\mathcal F$ is a foliation on  $M$ and $p\in M$ , {\em the multiplicity $\nu_p{\mathcal F}$ } is defined as the minimum of the multiplicities $\nu_p(a_i)$, $i=1,2,\ldots,n$, where $\mathcal F$ is locally given by $\omega=0$ with
$
\omega=\sum_{i=1}^na_idx_i
$
and the coefficients $a_i\in {\mathbb C}\{x_1,x_2,\ldots,x_n\}$ are without common factor.
\begin{remark}
If $\mathcal G$ is a germ of foliation in $({\mathbb C}^2,0)$ and $L=(x=0)$ is an invariant line, then $\mathcal G$ is given by $\alpha=0$ where $\alpha=a(x,y)dx+xb(x,y)dy$ and $a,xb$ are without common factors. The {\em restricted multiplicity} $\mu({\mathcal G}, L;0)$ is the order $\nu_y(a(0,y))$. Moreover, when ${\mathcal G}$ is a foliation on the projective plane ${\mathbb P}^2_{\mathbb C}$ and $L$ is a straight line invariant by $\mathcal G$ we have that
\begin{equation}
\label{eq:gradoinvariante}
d+1=\sum_{q\in L}\mu({\mathcal G}, L;q).
\end{equation}
where $d$ is {\em the degree} of ${\mathcal G}$ (see \cite{Can-C-D} for more details).
\end{remark}

We will do frequently arguments by induction on the {\em height $h(p)$} of a point $p\in \sigma_k^{-1}(0)$ with respect to the sequence $\mathcal S$. This number is defined by
\begin{equation*}
\label{eq:height}
h(p)=\sharp\{k'\geq k; p\in \pi_{k'k}(Y_{k'})\}.
\end{equation*}
\begin{lemma}
 \label{lema:CH-dicriticos}
 Assume that $Y_{b-1}=\{p\}$ is a single point and $\pi_b$ is a dicritical blowing-up. Let $\mathcal G$ be the restriction of ${\mathcal F}_b$ to the projective plane $E^b_b$ and let $d$ be the degree of $\mathcal G$. We have $d+1=\nu_p{\mathcal F}$ and moreover 
\begin{enumerate}
\item A point $q\in E^b_b$ is a pre-simple CH-corner for ${\mathcal F}_b, E^b$ if and only if it is a pre-simple CH-corner for ${\mathcal G}, D$, where $D=\cup_{i=1}^{b-1}E^b_b\cap E^b_i$.
    \item If $p$ is an s-trace point that belongs to a non dicritical component $E^{b-1}_i$ of $E^{b-1}$, there is an s-trace  point $q\in E^b_i\cap E^b_b$.
\end{enumerate}\end{lemma}
\begin{proof} Let ${\mathcal F}_{b-1}$ be locally given at $p$ by $\omega=0$ with $\omega=adx+bdy+cdz$, where $a,b,c\in {\mathbb C}\{x,y,z\}$ are without common factor. Put $r=\nu_p{\mathcal F}_{b-1}$ and write
$$
a=A_r+A_{r+1}+\cdots;\; b=B_r+B_{r+1}+\cdots ;\; c=C_r+C_{r+1}+\cdots
$$
the decomposition into homogeneous components. Since $\pi_b$ is dicritical, we have $XA_r+YB_r+ZC_r=0$. The foliation ${\mathcal G}$ is defined in the projective space $E^b_b$ by the global 1-form
$$
W=A_rdX+B_rdY+C_rdZ
$$
and in view of Remark \ref{rk:gradodicritico} the coefficients $A_r,B_r,C_r$ do not have common factor. This means that $d=r-1$.

Now, let $q\in E^b_b$ be a pre-simple CH-corner for ${\mathcal F}_b, E^b$. There are local coordinates $x,y,z$ at $q$ such that $E^b_b=(z=0)$ and we have one of the following cases
\begin{enumerate}
\item The point $q$ is non singular and the coordinates may be chosen such that ${\mathcal F}=(dx=0)$, $E^b_{\mbox{\small inv}}=(x=0)$ and $(z=0)\subset E^{b}_{\mbox{\small dic}}\subset (yz=0)$ locally at $q$.
    \item The point $q$ is  singular and the coordinates may be chosen such that
    $${\mathcal F}=\{(\lambda+f(x,y)) ydx+(\mu+g(x,y))x dy=0\},\; \lambda\mu\ne 0, \;\nu_q(f,g)\geq 1,$$ with
     $E^b_{\mbox{\small inv}}=(xy=0)$ and $(z=0)=E^{b}_{\mbox{\small dic}}$ locally at $q$.
\end{enumerate}
In both cases, we see that $q$ is a pre-simple corner for ${\mathcal G}, D$. Conversely, assume that $q$ is a pre-simple corner for ${\mathcal G}, D$. Take notations as in Equation \ref{eq:restriccion}, with $E^b_b=(z=0)$. If $q$ is non singular for $\mathcal G$, it is also non singular for ${\mathcal F}_b$ and we deduce that it is a pre-simple CH-corner just by looking at the positions of the divisors.  If it is singular, then $E^b=(xyz=0)$ locally at $q$ and
$$
{\mathcal G}=\{
a(x,y,0)dx+b(x,y,0)dy=(\lambda+f(x,y)) ydx+(\mu+g(x,y))x dy=0\}.
$$
The vector field $\xi=c(x,y,z)x{\partial/\partial x}-a(x,y,z){\partial/\partial z}$ trivializes the foliation ${\mathcal F}_b$ and we get a pre-simple CH-corner for ${\mathcal F}_b,E^b$.

Finally, assume that $p$ is an s-trace point belonging to a non dicritical component $E^{b-1}_i$ of $E^{b-1}$. We have three cases to consider:

{\em Case $e(E^{b-1}_{\mbox{\small inv}},p)=1$.}  The straight line $L=E^b_b\cap E^{b}_{\mbox{\small inv}}=E^b_b\cap E^{b}_{i}$ is  invariant by $\mathcal G$. By Equation \ref{eq:gradoinvariante} we have that $L$ contains $d+1=r\geq 1$ singular points of $\mathcal G$. Moreover, any point $q\in \mbox{Sing}{\mathcal G}\cap L$ is an s-trace point.

{\em Case $e(E^{b-1}_{\mbox{\small inv}},p)=2$.} We have  $E^b_b\cap E^{b}_{\mbox{\small inv}}=L_i\cup L_j$ with $ L_i= E^b_b\cap E^{b}_{i}$, $L_j= E^b_b\cap E^{b}_{j}$, where $E^{b-1}_i$ and $E^{b-1}_j$ are the two non dicritical components of $E^{b-1}$ containing $p$. Note that $L_i$ and $L_j$ are invariant lines for $\mathcal G$. If $d=0$, the intersection point $q_0\in L_i\cap L_j$ is the only singular point of $\mathcal G$, moreover it is a pre-simple CH-corner; this implies in view of Proposition \ref{pro:notchcorners} that $p$ is a pre-simple CH-corner, contradiction. Hence $d\geq 1$, in this case by Equation \ref{eq:gradoinvariante} we find at least one singular point in $L_i$ (and also in $L_j$) that is an s-trace point. To be precise, if $q_0$ is a pre-simple CH-corner, we have that  $\mu({\mathcal G}, L_i;q_0)=1$ and hence there is another singular point $q'$ in $L_i$ that must be an s-trace point.

{\em Case $e(E^{b-1}_{\mbox{\small inv}},p)=3$.} We have $E^b_b\cap E^{b}_{\mbox{\small inv}}=L_i\cup L_j\cup L_k$ with $ L_i= E^b_b\cap E^{b}_{i}$, $L_j= E^b_b\cap E^{b}_{j}$, $L_k= E^b_b\cap E^{b}_{k}$,
where $E^{b-1}_i$, $E^{b-1}_j$ and $E^{b-1}_k$ are the three non dicritical components of $E^{b-1}$ that contain $p$.  Note as before that $L_i$, $L_j$ and $L_k$ are invariant lines for $\mathcal G$ and there are three singular points $q_{ij}, q_{ik}, q_{jk}$ corresponding to the respective intersections of two lines. This implies that $d\geq 1$. If $d=1$ the points $q_{ij}, q_{ik}, q_{jk}$ are the only singular points of $\mathcal G$ and they are pre-simple CH-corners; thus $p$ must be a pre-simple CH-corner, contradiction. If $d\geq 2$, we find as before at least a point in $L_i$ (and also in $L_j$, $L_k$) that is an s-trace point.
\end{proof}
\begin{lemma}
 \label{lema:curvaesquina}
 Assume that $Y_{b-1}=E^{b-1}_i\cap E^{b-1}_j$ is a  generically pre-simple CH-corner curve and let $p$ be the intersection point $Y_{b-1}\cap \sigma_{b-1}^{-1}(0)$. Let us consider the point $\{q'_i\}=\pi_b^{-1}(p)\cap E^b_i$.  If  $p$ is an s-trace point, then $q'_i$ is also an s-trace point.
\end{lemma}
\begin{proof} Take local coordinates at $p$ such that $E^{b-1}_i=\{y=0\}$ and  $E^{b-1}_j=\{x=0\}$ and suppose that ${\mathcal F}_{b-1}$ is locally given at $p$ by $\omega=0$ where $\omega$ is the 1-form
$$
\omega=a(x,y,z)\frac{dx}{x}+b(x,y,z)\frac{dy}{y}+c(x,y,z)\frac{dz}{z^\epsilon}.
$$
If $e(E^{b-1}_{\mbox{\small inv}},p)=2$ we put $\epsilon=0$ and if  $e(E^{b-1}_{\mbox{\small inv}},p)=3$ we put $\epsilon =1$. Either way, $a,b,c$ are without common factors and in view of the hypothesis we have that
$$
a=\phi(z)+xf_1+yf_2; \; b=\psi(z)+xg_1+yg_2; \phi(z)\psi(z)\ne 0.
$$
Moreover, since $p$ is not a pre-simple CH-corner, we deduce that $\phi(0)=\psi(0)=0$, otherwise, we should contradict the CH character of the foliation ${\mathcal F}_{b-1}$ as in the proof of Proposition \ref{pro:notchcorners}.
In local coordinates $x,y'=y/x,z$ the foliation ${\mathcal F}_b$ is given at $q'_i$ by
$$
\omega=a'\frac{dx}{x^\delta}+b'\frac{dy'}{y'}+c'\frac{dz}{z^\epsilon}
$$
where $\delta=0$ in the dicritical case and $\delta=1$  if the blow-up is non dicritical. The coefficients
$a',b',c'\in {\mathbb C}\{x,y',z\}$ are without common factor and given by
$$
a'=x^{\delta-1}(a(x,xy',z)+b(x,xy',z));\; b'=b(x,xy',z);\;c'=c(x,xy',z).
$$
 We have that $b'(0,y',0)=0$ and thus $q'_i$ is not a pre-simple CH-point. Moreover, it is a singular point since the foliation ${\mathcal F}_b$ is locally given at $q'_i$ by the holomorphic form $\Omega=y'x^{\delta}z^{\epsilon}\omega$. Hence $q'_i$ is an s-trace point.
\end{proof}

\begin{proposition}
 \label{prop: singlocus}
 Let $p\in \sigma_k^{-1}(0)$  be an s-trace point for ${\mathcal F}_k,E^k$
  that belongs to a non dicritical component  $E^k_i$ of $E^k$.
There is an  s-trace curve $\Gamma$ such that
$
p\in\Gamma\subset E^k_i
$.
\end{proposition}

\begin{proof} We do induction on the height $h(p)$ of $p$. If $h(p)=0$, we are done, since $p$ is a simple CH-trace point. Assume that $h(p)\geq 1$. Let $b>k$ be the first index such that $p\in\pi_{k(b-1)}(Y_{b-1})$. We consider several cases.

{\em First case: the center $Y_{b-1}$ is a point $Y_{b-1}=\{p'\}$ and $\pi_{b}$ is non-dicritical}. We do the blow-up $\pi_{b}$ and by Proposition \ref{pro:notchcorners} there is a trace point $q\in \pi_{b}^{-1}(p')=E^{b}_{b}$. Since $E^{b}_{b}$ is compact (it is isomorphic to ${\mathbb P}^2_{\mathbb C}$) and invariant, we can
apply induction hypothesis to $q\in E_b^b$ to find a trace compact curve $\tilde\Gamma\subset E^{b}_{b}$. The curve $\tilde \Gamma$ intersects the projective line $E^{b}_i\cap E^{b}_{b}$ at least in a point $q'$ that must be a trace point.  We apply induction hypothesis to $q'\in E^b_i$ to find a trace curve  $\Gamma'\subset E^b_i$ such that $q'\in \Gamma'$ and we take $\Gamma=\pi_{kb}(\Gamma')$. (See Remark  \ref{rk:closednesstrace})

{\em Second case: $Y_{b-1}$ is a point $Y_{b-1}=\{p'\}$  and $\pi_{b}$ is a dicritical blow-up}. By Lemma \ref{lema:CH-dicriticos} we find a point $q'\in E^b_{i}$ that is a trace point and we proceed by induction as before.

 {\em Third case: $Y_{b-1}$ is a curve transversal to $E^{b-1}_i$}.  By Proposition \ref{pro:notchcorners}, there is a trace point $q$ in $\pi_b^{-1}(p')=E^{b}_b\cap E^b_i$, where $p'$ is the (only) point over $p$ such that $\pi_{k(b-1)}(p')=p$. We apply induction hypothesis to $q\in E^b_i$ at $q$ to obtain a trace curve $\Gamma'\subset E^b_i$ and we put $\Gamma=\pi_{kb}(\Gamma')$.

 The remaining situation is that $Y_{b-1}$ is a curve contained in $E^{b-1}_i$. If $Y_{b-1}$ is an s-trace curve, we are done by taking $\Gamma=\pi_{k(b-1)}(Y_{b-1})$. Otherwise we apply Lemma \ref{lema:curvaesquina} to proceed by induction.
 \end{proof}

\begin{proposition}
 \label{prop: singlocus2}
 Assume that the center $Y_{b-1}$ of $\pi_b$ is an s-trace curve. Then there is an s-trace curve $\Gamma\subset E^b_b$ such that $\pi_b(\Gamma)=Y_{b-1}$.
\end{proposition}
\begin{proof} It is enough to look at the generic point of $Y_{b-1}$ and to apply Proposition \ref{pro:notchcorners}. In this way we find at least one s-trace point over each generic point of $Y_{b-1}$ and thus we necessarily have at least an s-trace curve as stated.
\end{proof}

 \section{Nodal components for RICH-Foliations}
 \label{sec:nodalcomponents}
  Consider a RICH-foliation $\mathcal F$ in $({\mathbb C}^3,0)$ and fix an RI-reduction of singularities $\pi$ as in Equation \ref{eq:sucesionreduccion}. Take a nodal component $\mathcal N$ of $\pi^*{\mathcal F}, E$, where $E=E^N$ is the exceptional divisor of $\pi$. In this section we prepare the proof of Theorem \ref{teo:mainIII} by giving a list of structural properties of $\mathcal N$ at intermediate steps of the reduction of singularities assuming that $\mathcal N$ is compact and  does not intersect the union $E_{\mbox{\small dic}}$ of the dicritical components of $E$. In the next section we will find a contradiction with these properties.

 For any $0\leq k\leq N$, let us denote ${\mathcal N}_k=\rho_k({\mathcal N})$. We have that  ${\mathcal N}_k\subset \sigma_k^{-1}(0)$ and hence ${\mathcal N}_k$ is a connected and compact analytic subset of $E^k$. We have two possibilities: either ${\mathcal N}_k$ is a single point (in this case we put $s_k=0$) or it is a finite union of $s_k\geq 1$ compact irreducible analytic curves ${\mathcal N}_k=\Gamma^k_1\cup \Gamma^k_2\cup\cdots\cup\Gamma^k_{s_k}$. Let us remark that the curves $\Gamma^k_j\subset \sigma_k^{-1}(0)$ will never be used as a center of blow-up in the reduction of singularities. This implies that the generic points of $\Gamma^k_j$ are CH-simple for ${\mathcal F}_k, E^k$ and  only finitely many points in  $\Gamma^k_j$ will be modified by subsequent blow-ups. In particular ${\mathcal N}_{k+1}$ has the form
$$
{\mathcal N}_{k+1}=\Gamma^{k+1}_1\cup \Gamma^{k+1}_2\cup\cdots\cup\Gamma^{k+1}_{s_{k+1}}
$$
where $s_{k+1}\geq s_k$ and for each $1\leq j\leq s_k$ the curve $\Gamma^{k+1}_j$ is the strict transform of $\Gamma^k_j$ by $\pi_{k+1}$. The {\em date of birth} $b({\mathcal N})$ of $\mathcal N$ is the index such that
$
s_k=0
$
if $k<b({\mathcal N})$ and $s_{b({\mathcal N})}\geq 1$. Note that $1\leq b({\mathcal N})\leq N$.

We will give a list of results about the local behavior of ${\mathcal N}_k$  when $k\geq b({\mathcal N})$.

The following Lemma \ref{lema:cornoddic} shows that ${\mathcal N}_k$  has a behavior similar to ${\mathcal N}={\mathcal N}_N$ concerning the corners of the exceptional divisor.
\begin{lemma}[Non dicriticalness and nodality at corners]
\label{lema:cornoddic} Assume that
$k\geq b({\mathcal N})$.  Let
$p\in {\mathcal N}_k $ be locally the intersection of three components $E^k_i$, $ E^k_j$ and $E^k_\ell$ of $E^k$ and suppose that
$
E^k_i\cap E^k_j\subset {\mathcal N}^k
$. Then $E^k_i$, $ E^k_j$ and $E^k_\ell$ are invariant for ${\mathcal F}_k$ and
\begin{equation*}
E^k_i\cap E^k_\ell\subset {\mathcal N}_k \Leftrightarrow
E^k_\ell\cap E^k_j\not\subset {\mathcal N}_k.
\end{equation*}
(Equivalently, we have $E^k_i\cap E^k_\ell\not\subset {\mathcal N}_k \Leftrightarrow
E^k_\ell\cap E^k_j\subset {\mathcal N}_k$).
\end{lemma}
\begin{proof} We do induction on the height $h(p)$ of $p$.  If $h(p)=0$, the point $p$ is a simple CH-corner for ${\mathcal F}_k,E^k$ of nodal type that will not be modified by further blow-ups, so we can think locally at $p$ as in the case $k=N$. Since $\mathcal N$ does not intersect $E^N_{\mbox{\small dic}}$ the three components $E^k_i$, $ E^k_j$ and $E^k_\ell$ are invariant for ${\mathcal F}_k$ and $p$ is a simple CH-corner of dimensional type three.  The second assertion of the Lemma is a direct consequence of the observations in Remark
\ref{rk:transicionnodal}.

Now, assume that $h=h(p)\geq 1$. Let $b\geq k$ be the first index such that $p_b\in Y_b$, where $\pi_{bk}(p_b)=p$. Note that the local situation at $p_b$ is exactly the same one as the local situation at $p$. Let us consider the blow-up $\pi_{b+1}$. Let us note that $Y_b\ne E^b_i\cap E^b_j$, since the fact that $E^b_i\cap E^b_j\subset {\mathcal N}_b$ implies that $E^b_i\cap E^b_j$ is a compact curve and we do not use compact curves as centers in view of Remark \ref{rk:firstpropertiesS}.
 First, we see that $\pi_{b+1}$ is a non dicritical blow-up (that is, the component $E^{b+1}_{b+1}$ is invariant); otherwise, we apply induction hypothesis to the intersection point $p'$ in $E^{b+1}_i\cap E^{b+1}_j\cap E^{b+1}_{b+1}$ that is a point in ${\mathcal N}_{b+1}$ such that $E^{b+1}_i\cap E^{b+1}_j\subset {\mathcal N}_{b+1}$. Then we have the next possibilities to consider
 \begin{enumerate}
 \item The center of $\pi_{b+1}$ is a point $Y_b=\{p_b\}$.
 \item The center of $\pi_{b+1}$ is a point is the curve $Y_b=E^b_i\cap E^b_\ell$.
 \item The center of $\pi_{b+1}$ is a point is the curve $Y_b=E^b_j\cap E^b_\ell$.
 \end{enumerate}
The case (3) is like the case (2) just by interchanging the roles of the indices $i,j$.
 Consider the case (1). Define the points $p'_\ell$, $p'_j$ and $p'_i$ by
 \begin{eqnarray*}
 \{p'_\ell\}&=&E^{b+1}_i\cap E^{b+1}_j\cap E^{b+1}_{b+1};\\
 \{p'_j\}&=&E^{b+1}_i\cap E^{b+1}_\ell\cap E^{b+1}_{b+1};\\
 \{p'_i\}&=&E^{b+1}_j\cap E^{b+1}_\ell\cap E^{b+1}_{b+1}.
 \end{eqnarray*}
We apply induction hypothesis at $p'_\ell$ to see that $E^{b+1}_i$ and $E^{b+1}_j$ are invariant components of $E^{b+1}$, hence $E^{b}_i$ and $E^{b}_j$ are invariant components of $E^{b}$. Also by induction hypothesis applied at $p'_\ell$ one of the following properties holds
\begin{enumerate}
 \item[i)] $E^{b+1}_i\cap E^{b+1}_{b+1}\subset {\mathcal N}_{b+1}$ and  $E^{b+1}_j\cap E^{b+1}_{b+1}\not\subset {\mathcal N}_{b+1}$.
 \item[ii)] $E^{b+1}_i\cap E^{b+1}_{b+1}\not\subset {\mathcal N}_{b+1}$ and  $E^{b+1}_j\cap E^{b+1}_{b+1}\subset {\mathcal N}_{b+1}$.
 \end{enumerate}
 Assume we have i). We apply induction hypothesis at $p'_j$ to see that $E^{b+1}_\ell$ is invariant, hence $E^b_\ell$ is also invariant, and moreover one of the following properties holds
 \begin{enumerate}
 \item[a)] $E^{b+1}_i\cap E^{b+1}_{\ell}\subset {\mathcal N}_{b+1}$ and  $E^{b+1}_i\cap E^{b+1}_{b+1}\not\subset {\mathcal N}_{b+1}$.
 \item[b)] $E^{b+1}_i\cap E^{b+1}_{\ell}\not\subset {\mathcal N}_{b+1}$ and  $E^{b+1}_i\cap E^{b+1}_{b+1}\subset {\mathcal N}_{b+1}$.
 \end{enumerate}
 
 In  case a) we have $E^{b}_i\cap E^{b}_{\ell}\subset {\mathcal N}_{b}$. It remains to prove that $E^{b}_j\cap E^{b}_{\ell}\not\subset {\mathcal N}_{b}$. But if $E^{b}_j\cap E^{b}_{\ell}\subset {\mathcal N}_{b}$, we apply induction hypothesis at $p'_i$ and then either
 $E^{b+1}_i\cap E^{b+1}_{b+1}\subset {\mathcal N}_{b+1}$ or $E^{b+1}_\ell\cap E^{b+1}_{b+1}\subset {\mathcal N}_{b+1}$; in the first case we find a contradiction at the point $p'_\ell$ and in the second one we find a contradiction at $p'_j$. 
 
 In case b)  we have $E^{b}_i\cap E^{b}_{\ell}\not\subset {\mathcal N}_{b}$. Moreover, we can apply induction hypothesis at $p'_i$ to deduce that either $E^{b+1}_j\cap E^{b+1}_\ell\subset {\mathcal N}_{b+1}$ or
 $E^{b+1}_j\cap E^{b+1}_{b+1}\subset {\mathcal N}_{b+1}$, but in the second case we find a contradiction at $p'_i$ and hence we have that $E^{b}_j\cap E^{b}_\ell\subset {\mathcal N}_{b}$. 
 
 Finally, if we have ii) we do the same arguments as in i) by interchanging the indices $i,j$.

 Now, we consider the case (2) where the center of $\pi_{b+1}$ is the curve $Y_b=E^b_i\cup E^b_\ell$. Note that $Y_b$ is non compact and hence $Y_b\not\subset {\mathcal N}_{b}$.  Consider the points
 \begin{eqnarray*}
 \{q'_\ell\}&=&E^{b+1}_i\cap E^{b+1}_j\cap E^{b+1}_{b+1};\\
 \{q'_i\}&=&E^{b+1}_\ell\cap E^{b+1}_j\cap E^{b+1}_{b+1}.
 \end{eqnarray*}
 Applying induction at $q'_\ell$ we see that $E^{b+1}_i$, $E^{b+1}_j$, and $E^{b+1}_{b+1}$ are invariant components of $E^{b+1}$. Hence $E^{b}_i$, $E^{b}_j$ are invariant. Also by induction at $q'_\ell$  and since $E^{b+1}_{b+1}\cap E^{b+1}_i$ is non compact, we have that
 $E^{b+1}_{b+1}\cap E^{b+1}_j\subset {\mathcal N}_{b+1}$. Now, we apply induction at $q'_i$ to see that  $E^{b+1}_\ell$ is invariant and that $E^{b+1}_{j}\cap E^{b+1}_\ell\subset {\mathcal N}_{b+1}$ since
 $E^{b+1}_{\ell}\cap E^{b+1}_{b+1}$ is non compact. Hence $E^{b}_\ell$ is invariant and  $E^{b}_{j}\cap E^{b}_\ell\subset {\mathcal N}_{b}$.
\end{proof}
\begin{remark} In view of the terminology introduced in Remark \ref{rk:divisortrace} and considering the fact that a compact curve of the adapted singular locus is not contained in a dicritical component by Remark \ref{rk:gradodicritico}, we can reformulate Lemma \ref{lema:cornoddic} as follows:

 {``Let $\Gamma$ be a curve that is an irreducible component of ${\mathcal N}_k$ with $e(E^k,\Gamma)=2$. Then $\Gamma$ is generically a simple CH-corner and if $p$ is a point of intersection of $\Gamma$ with a  component $E^k_\ell$ of $E^k$ transversal to $\Gamma$, we have that $E^k_\ell$ is invariant and there is exactly one curve $\Gamma'\subset {\mathcal N}_k$ such that $\Gamma'\ne\Gamma$, $p\in \Gamma\cap \Gamma'$ and $\Gamma'$ is generically simple CH-corner.''
}
\end{remark}
\begin{definition}
\label{def:nodalinterruption}
Let $\Gamma$ be an s-trace curve for ${\mathcal F}_k, E^k$ with $e(E^k,\Gamma)\geq 1$. We say that $\Gamma$ is {\em interrupted by an irreducible component $E^k_\ell$ of $E^k$}  at a point $p$ if
$p\in \Gamma\cap E^k_\ell$ and  $\Gamma\not\subset E^k_\ell$. The interruption is of {\em nodal type} if $E^k_i\cap E^k_\ell\subset {\mathcal N}_k$ (locally at $p$), for any $E^k_i$ such that $\Gamma\subset E^k_i$.
\end{definition}

\begin{remark}
In the case
$e(E^k;\Gamma)=2$ with $\Gamma\subset E^k_i\cap E^k_j$ the curve $\Gamma$ is not compact since it is supposed to be an s-trace curve, see Remark  \ref{rk:divisortrace}. In particular $\Gamma\not\subset {\mathcal N}_k$. In this case, the following statements are equivalent
\begin{enumerate}
\item[a)] The interruption of $\Gamma$ by $E^k_\ell$ at $p$ is a nodal interruption.
\item[b)] $E^k_i\cap E^k_\ell\subset {\mathcal N}_k$.
\item[c)] $E^k_j\cap E^k_\ell\subset {\mathcal N}_k$.
\end{enumerate}
This is because by Lemma \ref{lema:cornoddic} we have that $E^k_i\cap E^k_\ell\subset {\mathcal N}_k$ implies that  $E^k_j\cap E^k_\ell\subset {\mathcal N}_k$ and conversely.  Also by Lemma \ref{lema:cornoddic}, in the case of a nodal interruption, all the concerned components $E^k_i,E^k_j,E^k_\ell$ are invariant ones.
\end{remark}

\begin{proposition} [Non dicriticalness and nodality at trace points]
\label{prop:ndicandnodality}
Consider a point $p\in {\mathcal N}_k$ with
$k\geq b({\mathcal N})$. Then all the components of $E^k$ through $p$ are non dicritical.  Moreover, let $\Gamma$ be an s-trace curve for ${\mathcal F}_k, E^k$ interrupted by  $E^k_\ell$ at $p$ such that $e(E^k,\Gamma)\geq 1$. There is an s-trace curve $\Gamma'\subset E^k_\ell$  with $p\in \Gamma'$ such that
\begin{enumerate}
\item If the interruption is of nodal type we have
$
\Gamma\subset {\mathcal N}_k\Leftrightarrow \Gamma'\not\subset {\mathcal N}_k
$.
\item If the interruption is not of nodal type we have
$
\Gamma\subset {\mathcal N}_k\Leftrightarrow \Gamma'\subset {\mathcal N}_k
$.
\end{enumerate}
\end{proposition}
\begin{proof} We proceed by induction on the height $h(p)$. If $h(p)=0$ we are done in view of the description o the singular locus at simple points and the hypothesis that $\mathcal N$ has only compact irreducible components and does not intersect the dicritical components of the divisor.

Assume that $h(p)\geq 1$. In order to simplify the notation we can assume that $p\in Y_k$, otherwise we consider the first index where this holds as in the proof of Lemma \ref{lema:cornoddic}. Also in order to simplify the writing, let us denote $b=k+1$.
Let us note that since $k\geq b({\mathcal N})$, there is at least one curve $\Delta\subset{\mathcal N}_k$ such that $p\in \Delta$. Note that $\Delta\subset E^k$ and $\Delta\ne Y_k$, because $\Delta$ is a compact curve. We will denote by $\Delta$ such curves if no confusion arises.

  First we prove that $\pi_b$ is a non-dicritical blow-up.
  We apply the induction hypothesis at a point $p'\in E^b_b\cap \Delta'$ where $\Delta'$ is  the strict transform  of $\Delta$ by $\pi_b$, this implies that $E^b_b$ is invariant, that is $\pi_b$ is a non dicritical blow-up.

A) Let us now prove that all the components of $E^k$ through $p$ are non dicritical.
 We have to consider the cases that $Y_k=\{p\}$ and $Y_k$ is a germ of curve at $p$.

{\em Case A-1: $Y_k=\{p\}$}. We recall that the divisor $E^b_b$ is invariant and isomorphic to a projective plane ${\mathbb P}_{\mathbb C}^2$.   We have two possibilities
\begin{enumerate}
\item [a)] There is an s-trace curve $\Delta$.
\item [b)] All the curves $\Delta$ are generically simple CH-corners.
\end{enumerate}
If we are in case a), take a point $p'\in \Delta'\cap E^b_b$, where $\Delta'$ is the strict transform of $\Delta$ by $\pi_b$.
We can apply induction hypothesis at $p'$, since $\Delta'$ is an s-trace curve and $\Delta'$ is interrupted by $E^b_b$ at $p'$.  The interruption may be of nodal type or not, in both cases we find a curve $\Delta''\subset E^b_b\cap {\mathcal N}_b$ with $p'\in \Delta''$. Now, given any component $E^k_\ell$ with $p\in E^k_\ell$, there is a point $p''$ belonging to the intersection of the projective line $E^b_\ell\cap E^b_b$ and $\Delta''$. By induction hypothesis at $p''$ we conclude that $E^k_\ell$ is invariant. If we are in case b), we take a generically simple CH-corner curve $\Delta$ with $p\in \Delta$ and we apply Lemma \ref{lema:cornoddic} at the point $p'\in \Delta'\cap E^b_b$ to find a curve $\Delta''\subset E^b_b\cap {\mathcal N}_b$ as before. We conclude as in case a) that any component of $E^k$ through $p$ is invariant.

{\em Case A-2: $Y_k$ is a germ of curve at $p$}.
Note that $e(E^k,p)\geq 1$ since $\Delta\subset E^k$ (this is also valid for the previous case). Suppose that $\Delta\subset E^k_i$, taking a point $p'\in E^b_b\cap \Delta'$ we have that $p'\in{\mathcal N}_b\cap E^b_i$. By induction $E^b_i$ and hence $E^k_i$ are invariant components. Thus, it is enough to look at the components $E^k_\ell$ such that $p\in E^k_\ell$ and $\Delta\not\subset E^k_\ell$ for any  $\Delta$, in particular, we can suppose that $e(E^k,p)\geq 2$.

 We have several possibilities
\begin{enumerate}
\item $e(E^k,Y_k)=1$, $e(E^k,p)=2$. Put $Y_k\subset E^k_i$ and $p\in E^k_i\cap E^k_j$.  If $\Delta\subset E^k_i$, the point $p'\in \Delta'\cap E^b_b$ belongs to $E^b_i\cap E^b_j$ and  both $E^k_i$ and $E^k_j$ are invariant components. Assume now that $\Delta\subset E^k_j$ but $\Delta\not\subset E^k_i$, in particular $\Delta$ is an s-trace curve. We consider a point $p'\in \Delta'\cap E^b_b$ and we are going to apply induction at $p'$. We have two possible situations
    \begin{enumerate}
    \item $p'\in E^b_i$. Then $E^k_i$ and $E^k_j$ are invariant components as above.
    \item $p'\notin E^b_i$. Since $p'\in \Delta'\subset E^b_j$, we have that $E^k_j$ is an invariant component. Moreover, the s-trace curve $\Delta'$ is interrupted at $p'$ by $E^b_b$ and we can apply induction. The only compact curve through $p'$ contained in $E^b_b$ is $E^b_b\cap E^b_j$. If the interruption is not a nodal one, we should have a compact curve $\Gamma'\subset {\mathcal N}_b\cap E^b_b$ different from $E^b_b\cap E^b_j$. This is not possible, then we have a nodal interruption and thus $\Delta''\subset {\mathcal N}_b$, where $\Delta''=E^b_b\cap E^b_j$. Consider the point $p''\in \Delta''\cap E^b_i$. By induction hypothesis at $p''$ we deduce that $E^b_i$ and hence $E^k_i$ are invariant components.
    \end{enumerate}
\item $e(E^k,Y_k)=e(E^k,p)=2$.  Put $Y_k\subset E^k_i\cap E^k_j$. Note that $\Delta\ne E^k_i\cap E^k_j$ and thus, up to a reordering of the indices we have a curve $\Delta$ with $\Delta\subset E^k_i$ and $\Delta\not\subset E^k_j$. We deduce as above that $E^k_i$ is invariant. Let $p'$ be the point given by $p'\in \Delta'\cap E^b_b$, that is $\{p'\}= \pi_b^{-1}(p)\cap E^b_i$. The only compact curve contained in $E^b_b$ through $p'$ is $\Delta''=\pi_b^{-1}(p)$, by applying induction at $p'$ we conclude that $\Delta''\subset {\mathcal N}_b$. Moreover  $\Delta''$ is interrupted in a non nodal way by $E^b_j$ at $p''\in E^b_j\cap \Delta''$. In particular $p''\in {\mathcal N}_b$ and by induction $E^k_j$ is an invariant component. (By the way, we find an s-trace curve $\Delta'_1\subset E^b_j\cap {\mathcal N}_b$ and thus there is also a compact curve $\Delta_1\subset {\mathcal N}_k$ with $p\in \Delta_1\subset E^k_j$).
\item $e(E^k,Y_k)=2$, $e(E^k,p)=3$.  Put $Y_k\subset E^k_i\cap E^k_j$ and $p\in E^k_i\cap E^k_j\cap E^k_\ell$.  Denote $\Delta''=\pi^{-1}(p)$. We are going to show that $\Delta''\subset {\mathcal N}_b$, then we conclude by induction applied at  $\Delta''\cap E^b_i$ and $\Delta''\cap E^b_j$ that $E^k_i$ and $E^k_j$ are invariant components, moreover since
     $\Delta''\subset E^b_\ell$ we also conclude that $E^k_\ell$ is an invariant component. Now, if $\Delta$ is an s-trace curve, we conclude as in the previous cases that $\Delta''\subset{\mathcal N}_b$.  Finally, if $\Delta\subset E^k_i\cap E^k_\ell$, we apply Proposition \ref{pro:notchcorners} to conclude that $\Delta''\subset {\mathcal N}_b$.
\end{enumerate}

B)  Now, let $\Gamma$ be an s-trace curve interrupted by $E^k_\ell$ at $p$ such that $e(E^k,\Gamma)\geq1$.  If $e(E^k,\Gamma)=1$ we put $\Gamma\subset E^k_i$ and if $e(E^k,\Gamma)=2$ we put $\Gamma\subset E^k_i\cap E^k_j$ and hence $\Gamma=E^k_j\cap E^k_\ell$, locally at $p$. We denote by $\tilde\Gamma$ the strict transform of $\Gamma$ (in the cases that $\Gamma\ne Y_k$) and by $\tilde p$ a point in $\tilde\Gamma\cap E^b_b$. Let us also consider a curve $\Delta\subset{\mathcal N}_k$ with $p\in \Delta$, denote by $\Delta'$ the strict transform of $\Delta$ and take a point $q'\in \Delta'$.

{\em B-0) Case $Y_k=\Gamma$}. Note that in this case we have that $\Gamma\not\subset{\mathcal N}_k$. If all the points in $\pi_b^{-1}(p)\cap \mbox{Sing}({\mathcal F}_b,E^b)$ are pre-simple CH-corners, we deduce that $p$ is also a pre-simple CH-corner by Proposition \ref{pro:notchcorners}, but this is not possible since $\Gamma=Y_k$ is an s-trace curve. Otherwise, there is at least one s-trace point  $r'\in \pi_b^{-1}(p)$. We apply Proposition \ref{prop: singlocus} at $r'$ to find a trace curve $\tilde \Gamma\subset E^b_b$ and $\tilde\Gamma'\subset E^b_\ell$ with $r'\in \tilde \Gamma$. Now, $\tilde\Gamma$ is not compact and hence $\tilde\Gamma\not\subset{\mathcal N}_b$. Moreover, if the interruption of $\Gamma=Y_k$ at $p$ is nodal, we find by induction that $\pi_b^{-1}(p)\subset {\mathcal N}_b$ and thus the interruption of $\tilde\Gamma$ at $r'$ by $E^b_\ell$ is also a nodal interruption. By induction, we find a trace curve $\tilde\Gamma'\subset E^b_\ell$ with $r'\in \tilde\Gamma'$ such that $\tilde\Gamma'\subset {\mathcal N}_b$. By projection of $\tilde\Gamma'$ we find $\Gamma'\subset {\mathcal N}_k$. If the interruption of $\Gamma$ at $p$ is not nodal, we also find that $\pi_b^{-1}(p)\not\subset {\mathcal N}_b$ and thus the interruption of $\tilde\Gamma$ at $r'$ by $E^b_\ell$ is also a not nodal interruption. By induction we find $\Gamma'\not\subset{\mathcal N}_k$
as above.

{\em B-1) Case $Y_k=\{p\}$}. Assume first that $e(E^k,p)=2$ with $p\in E^k_i\cap E^k_\ell$. If $\tilde p\in E^b_i\cap E^b_\ell\cap E^b_b$ we can apply simultaneously induction at $\tilde p$ and Proposition \ref{pro:notchcorners}  to see that there is an s-trace curve $\tilde\Gamma'$ with $\tilde p\in \tilde\Gamma'\subset {E^b_\ell}$ such that
\begin{eqnarray*}
\mbox{If } E^b_i\cap E^b_\ell\subset {\mathcal N} _b \mbox{ then } (\tilde \Gamma\subset {N}_b \Leftrightarrow \tilde \Gamma'\not\subset {N}_b)\\
\mbox{If } E^b_i\cap E^b_\ell\not\subset {\mathcal N} _b \mbox{ then } (\tilde \Gamma\subset {N}_b \Leftrightarrow \tilde \Gamma'\subset {N}_b)
\end{eqnarray*}
  If $\tilde p\in E^b_i\cap E^b_b\setminus E^b_\ell$, by induction applied at $\tilde p$ we find an s-trace curve $\tilde\Gamma_1$ with $\tilde p\in \tilde \Gamma_1\subset E^b_b$ that cuts $E^b_\ell\cap E^b_b$ in a point $\tilde p_1$ and hence it is interrupted at $\tilde p_1$ by $E^b_\ell$, this implies the existence of an s-trace curve $\tilde\Gamma'$ with $\tilde p_1\in \tilde\Gamma'\subset E^b_\ell$ such that
\begin{enumerate}
\item
\mbox{If } $E^b_i\cap E^b_\ell\subset {\mathcal N} _b$ \mbox{ then } $(E^b_i\cap E^b_b\subset {N}_b \Leftrightarrow E^b_\ell\cap E^b_b\not\subset {N}_b)$  and, doing an argument through $\tilde \Gamma_1$ we have $(\tilde \Gamma\subset {N}_b \Leftrightarrow \tilde \Gamma'\not\subset {N}_b)$.
\item \mbox{If } $E^b_i\cap E^b_\ell\not\subset {\mathcal N} _b$ \mbox{ then } $(E^b_i\cap E^b_b\subset {N}_b \Leftrightarrow E^b_\ell\cap E^b_b\subset {N}_b)$  and, doing an argument through $\tilde \Gamma_1$ we have $(\tilde \Gamma\subset {N}_b \Leftrightarrow \tilde \Gamma'\subset {N}_b)$.
\end{enumerate}
We end by projecting $\tilde \Gamma'$ by $\pi_b$ to obtain $\Gamma'$ and noting that the interruption is a nodal one if and only if  $E^k_i\cap E^k_\ell\subset {\mathcal N}_k$. Finally, the case that $e(E^k,p)=3$ is done with the same arguments as in the situation with $e(E^k,p)=2$.

{\em B-2) Case that $Y_k$ is a curve with $Y_k\ne \Gamma$}. Since $e(E^k,p)\geq 2$ and $Y_k$ has normal crossings with $E^k$ we have that  $Y_k\subset E^k_i$ or $Y_k\subset E^k_\ell$. If $Y_k\subset E^k_i$ but $Y_k\not\subset E^k_\ell$ we have that $\pi_b^{-1}(p)\subset E^b_\ell$ and we can apply induction at $\tilde p$ to see that there is an s-trace curve $\tilde\Gamma'$ with $\tilde p\in \tilde\Gamma'\subset E^b_\ell$ such that
\begin{eqnarray*}
\mbox{If } E^b_i\cap E^b_\ell\subset {\mathcal N} _b \mbox{ then } (\tilde \Gamma\subset {N}_b \Leftrightarrow \tilde \Gamma'\not\subset {N}_b)\\
\mbox{If } E^b_i\cap E^b_\ell\not\subset {\mathcal N} _b \mbox{ then } (\tilde \Gamma\subset {N}_b \Leftrightarrow \tilde \Gamma'\subset {N}_b).
\end{eqnarray*}
We end as in the previous cases of B-1). Assume that $Y_k\subset E^k_\ell$ but $Y_k\not\subset E^k_i$. Now we have that $\pi_b^{-1}(p)=E^b_i\cap E^b_b$. By induction at $\tilde p$ there is an s-trace curve $\tilde\Gamma'$ with $\tilde p\in \tilde\Gamma'\subset E^b_\ell$ such that
\begin{eqnarray*}
\mbox{If } \pi_b^{-1}(p)\subset {\mathcal N}_b \mbox{ then } (\tilde \Gamma\subset {N}_b \Leftrightarrow \tilde \Gamma'\not\subset {N}_b)\\
\mbox{If } \pi_b^{-1}(p)\not\subset {\mathcal N}_b \mbox{ then } (\tilde \Gamma\subset {N}_b \Leftrightarrow \tilde \Gamma'\subset {N}_b).
\end{eqnarray*}
Let $q$ be the point of intersection of $E^b_\ell$ and $\pi_b^{-1}(p)$. Note that $E^b_b\cap E_\ell$ is not compact and thus
$E^b_b\cap E_\ell\not\subset {\mathcal N}_b$. Thus, by Proposition \ref{pro:notchcorners}, we have that
$$
 \pi_b^{-1}(p)\subset {\mathcal N}_b  \Leftrightarrow E^b_i\cap E^b_\ell\subset {N}_b.
$$
We conclude that
\begin{eqnarray*}
\mbox{If } E^b_i\cap E^b_\ell\subset {\mathcal N}_b \mbox{ then } (\tilde \Gamma\subset {N}_b \Leftrightarrow \tilde \Gamma'\not\subset {N}_b)\\
\mbox{If } E^b_i\cap E^b_\ell\not\subset {\mathcal N}_b \mbox{ then } (\tilde \Gamma\subset {N}_b \Leftrightarrow \tilde \Gamma'\subset {N}_b).
\end{eqnarray*}
We end this case by noting that the interruption is nodal if and only if $E^k_i\cap E^k_\ell\not\subset {\mathcal N}_k$.

It remains to consider the case that $Y_k=E^k_i\cap E^k_\ell$ locally at $p$. The interruption then is a non nodal one, since $Y_k\not\subset{\mathcal N}_k$. If $Y_k$ is an s-trace curve, we are done by taking $\Gamma'=Y_k$. Assume that  $Y_k$ is generically pre-simple CH-corner. By induction at $\tilde p$, there is an s-trace curve $\tilde\Gamma_1\subset E^b_b$ such that
$$
\tilde\Gamma_1\subset{\mathcal N}_b\Leftrightarrow \tilde\Gamma\subset {\mathcal N}_b.
$$
We necessarily have that $\tilde\Gamma_1=\pi_b^{-1}(p)$ since there are not other possible singular curves over a generically pre-simple CH-corner like $Y_k$ after a nondicritical blow-up. We apply induction at the point of intersection of $\pi^{-1}(p)$ with $E^b_\ell$ to obtain $\tilde\Gamma'\subset E^b_\ell$ with
$$
\tilde\Gamma\subset{\mathcal N}_b\Leftrightarrow \tilde\Gamma_1\subset {\mathcal N}_b\Leftrightarrow \tilde\Gamma'\subset {\mathcal N}_b
$$
since the interruption of $\pi^{-1}(p)=\tilde\Gamma_1$ by $E^b_\ell$ is not nodal because $E^b_\ell\cap E^b_b$ is not compact and thus $E^b_\ell\cap E^b_b\not\subset {\mathcal N}_b$. This ends the proof.
\end{proof}
\begin{remark}
\label{rk:libre}
It is not necessary to work at a point $p\in {\mathcal N}_k$ to obtain conclusion (2) of Proposition \ref{prop:ndicandnodality}. To be precise, the following statement is also true as a direct consequence of Proposition \ref{prop: singlocus}:

{\em
``Let $\Gamma$ be an s-trace curve for ${\mathcal F}_k, E^k$ and suppose that there is an invariant component $E^k_\ell$ with $\Gamma\not\subset E^k_\ell$ and $p$ is a point $p\notin {\mathcal N}_k$ with $p\in \Gamma\cap E^k_\ell$.  There is an s-trace curve $\Gamma'\subset E^k_\ell$  with $p\in \Gamma'$ such that $\Gamma'\not\subset {\mathcal N}_k$''.
}
\end{remark}
\begin{proposition}[Incompatibility of trace curves]
 \label{pro:incompatibilidad}
Consider two s-trace curves $\Gamma_1$ and $\Gamma_2$ having a common point  $p\in \Gamma_1\cap\Gamma_2$ and contained in a common component $E^k_i$ of $E^k$. Then $\Gamma_1\subset{\mathcal N}_k$ if and only if $\Gamma_2\subset{\mathcal N}_k$. In an equivalent way, it is not possible that $\Gamma_1\subset{\mathcal N}_k$ and $\Gamma_2\not\subset{\mathcal N}_k$.
\end{proposition}
\begin{proof} Induction on the height $h(p)$. If $h(p)=0$ we are done, since there is at most one s-trace curve through $p$ contained in a component of the divisor. Assume that $h(p)\geq 1$ and suppose that $\Gamma_1\subset{\mathcal N}_k$ and $\Gamma_2\not\subset{\mathcal N}_k$ in order to find a contradiction. Take notations and conventions as in the proof of Proposition \ref{prop:ndicandnodality}. If $Y_k=\{p\}$, we have that $\tilde \Gamma_1$ and $\tilde \Gamma_2$ are interrupted by $E^b_b$ and the interruption given by $E^b_i\cap E^b_b$ is simultaneously nodal or not nodal for $\tilde \Gamma_1$ and $\tilde \Gamma_2$. By Proposition \ref{prop:ndicandnodality} and Remark \ref{rk:libre} we obtain two trace curves $\tilde \Gamma'_1$ and $\tilde \Gamma'_2$ contained in $E^b_b$ that must have a common point $p'\in \tilde\Gamma'_1\cap\tilde\Gamma'_2$, where we find a contradiction by induction hypothesis.

In the case $Y_k$ is a germ of curve $p\in Y_k$ and $Y_k\subset E^k_i$, with $Y_k\ne \Gamma_2$, (note that we know that $Y_k\ne\Gamma_1$) we find directly a contradiction by looking at $\tilde \Gamma_1$ and $\tilde \Gamma_2$ that are contained in $E^b_i$ and have the common point $p'=\pi^{-1}(p)\cap E^b_i$.

Assume that $Y_k=\Gamma_2$. We know that $E^b_i\cap E^b_b\not\subset{\mathcal N}_b$. Take the point $q$ such that $q\in \pi_b^{-1}(p)\cap E^b_i$. The strict transform $\tilde\Gamma_2$ is interrupted in a not nodal way by $E^b_b$ at $q$ and applying Proposition \ref{prop:ndicandnodality} we find a compact s-trace curve $\Delta\subset E^b_b\cap {\mathcal N}_b$. The only possibility is that $\Delta=\pi_b^{-1}(p)$. Moreover, by Proposition \ref{prop: singlocus2} there is an s-trace curve $\gamma\subset E^b_b$ that must intersect $\Delta$ in at least one point $r$. We find a contradiction at the point $r$ by applying induction hypothesis.

Assume that $Y_k$ is a germ of curve $p\in Y_k$ and $Y_k\not\subset E^k_i$. If $\pi^{-1}(p)\subset {\mathcal N}_b$ we find a contradiction since the nodal interruption of ${\tilde\Gamma_2}$ produces a trace curve of ${\mathcal N}_b$ contained in $E^b_b$ different from the only compact curve $\pi^{-1}(0)$. If $\pi^{-1}(p)\not\subset {\mathcal N}_b$ we find a contradiction since the nodal interruption of ${\tilde\Gamma_1}$ produces a trace curve of ${\mathcal N}_b$ contained in $E^b_b$ different from the only compact curve $\pi^{-1}(0)$. This ends the proof.
\end{proof}
\section{Conclusion}
In this section we end the proof of Theorem \ref{teo:mainIII}. As in Section \ref{sec:nodalcomponents}, we consider
 a RICH-foliation $\mathcal F$ in $({\mathbb C}^3,0)$  jointly with an RI-reduction of singularities
$
\pi:(M,\pi^{-1}(0))\rightarrow ({\mathbb C}^3,0)
$.
We will find a contradiction with the existence of a nodal component $\mathcal N$ which is compact and does not intersect the union  of the dicri\-tical components of the exceptional divisor $E$.

Let us go to the date of birth $b=b({\mathcal N})$. Put $k=b-1$. Note that $b\geq 1$. Consider first the situation that $b=1$. We consider the possibilities that $Y_0=\{0\}$ or $Y_0$ is a germ of curve.

 Assume that $Y_0=\{0\}$. We have an s-trace compact curve $\Gamma\subset E^1_1$ such that $\Gamma\subset {\mathcal N}_1$. In particular $E^1_1$ is invariant by Proposition \ref{prop:ndicandnodality}.
 Now, take a plane $(\Sigma,0)\subset ({\mathbb C}^3,0)$ that induces a transversal section of $\mathcal F$ in the sense of Mattei-Moussu \cite{Mat-M}. We know that $\Sigma$ may be chosen generic enough to assure that the strict transform $\tilde\Sigma$ of $\Sigma$ by $\pi_1$ cuts transversely $\mbox{Sing}({\mathcal F}_1,E^1)$ only at simple points (the set of non simple points in $E^1_1$ is finite). Thus, the restriction $$\sigma: (\tilde\Sigma, \tilde\Sigma\cap E^1_1)\rightarrow (\Sigma,0) $$
 of $\pi_1$
 provides a reduction of singularities of ${\mathcal G}={\mathcal F}\vert_{\Sigma}$. Let $p\in \Gamma\cap \tilde\Sigma$. The Camacho-Sad index (see \cite{Cam-S}) of $\mathcal G$ at $p$ is a positive real number, since $p$ is a simple point of nodal type for $\mathcal F$ and hence for $\mathcal G$. Since the sum of indices is $-1$, there is another point $q\in \tilde\Sigma\cap E^1_1$ which is a simple not nodal point for $\mathcal G$ and hence for $\mathcal F$. This implies the existence of an s-trace curve $\Gamma_1\subset E^1_1$ that is not generically nodal and hence $\Gamma_1\not\subset {\mathcal N}_1$. Since $E^1_1$ is isomorphic to a projective plane, there is a common point $r\in \Gamma\cap \Gamma_1$ in contradiction with Proposition \ref{pro:incompatibilidad}.

 Assume  now that $Y_0$ is a germ of curve. Then $\Gamma=\pi_1^{-1}(0)$ is the only possible compact curve in ${\mathcal N}_1$ and we also have that $E^1_1$ is invariant.
 Looking at two dimensional transversal sections of the center $Y_0$ at generic points and recalling that a non-dicritical blow-up in dimension two produces at least one singular point (given for instance by Camacho-Sad separatrix, see \cite{Cam-S, Can-C-D}), we find at least one singular curve $\Gamma_1\subset E^1_1$ which projects onto $Y_0$. Then $\Gamma_1$ and $\Gamma$ also give a contradiction with Proposition \ref{pro:incompatibilidad}.

 Assume that $b>1$. Let us suppose first that the center of $\pi_k$ is a point $Y_k=\{p\}$. We have a compact curve $\Delta\subset E^b_b$  such that $\Delta\subset {\mathcal N}_b$. By Proposition \ref{prop:ndicandnodality} and Lemma \ref{lema:cornoddic} we conclude that $E^b_b$ is invariant (that is the blow-up $\pi_k$ is non-dicritical) and also all the components $E^k_i$ through $p$ are invariant, since $\Delta$ cuts each $E^b_i$ such that $p\in E^k_i$. Let us note that $e(E^k,p)\geq 1$, since $b>1$ in particular there is at least one $E^k_i$ with $p\in E^k_i$. Let us prove the following statement
 \begin{quote}
 ``For all $E^k_i$ with $p\in E^k_i$ we have that $\Delta'=E^b_i\cap E^b_b\subset {\mathcal N}_b$.''
 \end{quote}
In the case that $\Delta=E^b_i\cap E^b_b$  we are done. Otherwise $\Delta$ cuts in at least one point $q$ the intersection $E^b_k\cap E^b_b$. If $\Delta$ is not an s-trace curve, we have $\Delta=E^b_j\cap E^b_b$ and we can apply Lemma \ref{lema:cornoddic} to see that either $E^b_i\cap E^b_j$ or $E^b_i\cap E^b_b$ is a curve in ${\mathcal N}_b$. But we know that $E^b_i\cap E^b_j\not\subset{\mathcal N}_b$, since otherwise $E^k_i\cap E^k_j\subset{\mathcal N}_k$ and $k$ is strictly smaller that the date of birth of ${\mathcal N}$. If $\Delta$ is  an s-trace curve, we
apply Proposition \ref{prop:ndicandnodality} at $q$, if the interruption of $\Delta$ by $E^b_i$ is a nodal one, we are done, since this means that $E^b_i\cap E^b_b\subset {\mathcal N}_b$, if it is a not nodal interruption, we find an s-trace curve $\tilde\Delta''\subset {\mathcal N}_b\cap E^b_i$ that projects onto an s-trace curve $\Delta''\subset {\mathcal N}_k\cap E^k_i$ and this also contradicts the fact that $k$ is strictly smaller that the date of birth of ${\mathcal N}$.

Now, we take a transversal two dimensional section $\Sigma$ at $p$ as in the case $b=1$. We find a point $q\in \tilde\Sigma\cap E^b_b$ that is a simple not nodal point for $\mathcal G$; moreover, the point $q$ is outside the intersections $\tilde\Sigma\cap E^b_b\cap E^b_i$ for each $E^b_i$, since these points are nodal ones. In this way we discover an s-trace curve $\Gamma_1\subset E^b_b$ which is not generically nodal and hence $\Gamma_1\not\subset {\mathcal N}_b$. Take a point $r\in \Gamma_1\cap E^b_i$. We apply Proposition \ref{prop:ndicandnodality} at $r$, since the interruption of $\Gamma_1$ by $E^b_i$ at $r$ is a nodal one, there is an s-trace curve $\tilde \Gamma_2\subset {\mathcal N}_b\cap E^b_i$ that must project onto an s-trace curve $\Gamma_2\subset{\mathcal N}_k\cap E^k_i$. We obtain in this way a contradiction as above.

In order to end the proof, let us suppose that $b>1$ and $Y_k$ is a germ of curve with $\{p\}=Y_k\cap E^k$. The only new compact curve after blow-up is $\Delta=\pi_b^{-1}(p)$ and hence we have that $\Delta\subset {\mathcal N}_b$. By a direct computation as in the case $b=1$, we find a singular curve $\gamma\subset E^b_b$ that projects onto $Y_k$. Note that $\gamma\not\subset {\mathcal N}_b$, since it is not a compact curve. Let $q$ be a point $q\in \gamma\cap \Delta$. Looking at the point $q$, we obtain a contradiction as follows
\begin{enumerate}
\item If $\Delta$, $\gamma$ are both  s-trace curves, we apply the incompatibility result of Proposition \ref{pro:incompatibilidad}.
\item If $\Delta$ is an  s-trace curve, but $\gamma$ is not, then $\gamma=E^b_b\cap E^b_i$ and we find by Proposition  \ref{prop:ndicandnodality} an s-trace curve $\tilde\Gamma\subset {\mathcal N}_b\cap E^b_i$ that projects onto an s-trace curve $\Gamma\subset {\mathcal N}_k\cap E^k_i$, contradiction.
\item  If $\Delta$ is not an s-trace curve, but $\gamma$ is, then $\Delta=E^b_b\cap E^b_i$ and we find a contradiction as in the preceding case.
\item  If $\Delta$ and $\gamma$ are both generically pre-simple corner curves then $\Delta=E^b_b\cap E^b_i$ and
 $\gamma=E^b_b\cap E^b_j$ and by Lemma \ref{lema:cornoddic}, we deduce that $E^b_i\cap E^b_j\subset {\mathcal N}_b$ and hence $E^k_i\cap E^k_j\subset {\mathcal N}_k$, contradiction.
\end{enumerate}
This finishes the proof of Theorem \ref{teo:mainIII}.

Now, Theorem \ref{teo:mainII} is a consequence of  Theorem \ref{teo:mainI}  and Theorem \ref{teo:mainIII} as follows. Property (ii) of Theorem \ref{teo:mainII} occurs when we have a nodal component $\mathcal N$ which is neither compact nor cuts a non compact dicritical component. So, if there are nodal components, we can assume that all of them intersect at least one compact dicritical component of the exceptional divisor. In this situation we can cover the nodal components with leaves containing germs of analytic curves at the origin and the arguments of Theorem \ref{teo:mainI} apply.

\end{document}